\documentclass[a4paper,11pt,fullpage]{article}
\usepackage[english]{babel}

\usepackage[left=1in,right=1in,top=1in,bottom=1in]{geometry}

\usepackage{amsmath}
\usepackage{amsthm}
\usepackage{amsfonts}
\usepackage{amssymb}
\usepackage{amsxtra}
\usepackage{latexsym}
\usepackage{ifthen}
\usepackage{cite}
\usepackage{esint}
\usepackage[cp1250]{inputenc}

\usepackage{epic}
\usepackage{eepic}

\DeclareMathOperator*{\esssup}{ess\,sup}

\DeclareMathOperator*{\osc}{osc}

\numberwithin{equation}{section}
\newtheorem{theorem}{Theorem}[section]
\newtheorem{lemma}{Lemma}[section]
\newtheorem{remark}{Remark}[section]

\def\XXint#1#2#3{{\setbox0=\hbox{$#1{#2#3}{\int}$}
     \vcenter{\hbox{$#2#3$}}\kern-.5\wd0}}

\begin{document}

\title{ Continuity and Harnack inequalities for local minimizers of non uniformly elliptic functionals with generalized Orlicz growth under the non-logarithmic conditions
\thanks{Dedicated to Dr. Mykhailo Voitovych, missing in the hero-city of Mariupol}
}

\author{ Maria O. Savchenko, Igor I. Skrypnik, Yevgeniia A.Yevgenieva
 }

  \maketitle

  \begin{abstract} We study the qualitative properties of functions belonging to the corresponding De Giorgi classes
	\begin{equation*}
	\int\limits_{B_{r(1-\sigma)}(x_{0})}\,\varPhi(x, |\nabla(u-k)_{\pm}|)\,dx
	\leqslant \gamma\,\int\limits_{B_{r}(x_{0})}\,\varPhi\bigg(x, \frac{(u-k)_{\pm}}{\sigma r}\bigg)\,dx,
	\end{equation*}
where $\sigma$, $r \in (0,1)$, $k\in \mathbb{R}$ and the function
$\varPhi$ satisfies the non-logarithmic condition
\begin{equation*}
\bigg(r^{-n}\int\limits_{B_{r}(x_{0})}[\varPhi\big(x,\frac{v}{r}\big)]^{s}\,dx\bigg)^{\frac{1}{s}}\bigg(r^{-n}\int\limits_{B_{r}(x_{0})}[\varPhi\big(x,\frac{v}{r}\big)]^{-t}\,dx\bigg)^{\frac{1}{t}}\leqslant c(K) \Lambda(x_{0},r),\quad r\leqslant v\leqslant K\,\lambda(r),
\end{equation*}
under some assumptions on the functions $\lambda(r)$ and $\Lambda(x_{0}, r)$ and the numbers $s$, $t >1$. These conditions generalize the known
logarithmic, non-logarithmic and non uniformly elliptic conditions.

In particular, our results cover new cases of non uniformly elliptic double-phase, degenerate double-phase functionals  and functionals with variable exponents.




\end{abstract}

\pagestyle{myheadings} \thispagestyle{plain}
\markboth{Maria O. Savchenko, Igor I. Skrypnik, Yevgeniia A.Yevgenieva}
{Continuity and Harnack inequalities....}

\section{Introduction and main results}\label{Introduction}

To explain the point of view of this research consider the  energy integrals $\int\limits_{\Omega}\, \varPhi_{i}(x,|\nabla u|)dx,$\\
$\varPhi_{1}(x, v)= v^{p}+ a_{1}(x) v^{q},\quad a_{1}(x)\geqslant 0,\quad \quad \osc\limits_{B_{r}(x_{0})}a_{1}(x)\leqslant A\,r^{q-p},\quad A  >0,\quad v > 0.$\\
$\varPhi_{2}(x, v)= v^{p} \big(1+a_{2}(x) \log(1+v) \big), \quad \quad a_{2}(x)\geqslant 0,\quad \quad \osc\limits_{B_{r}(x_{0})}a_{2}(x)\leqslant \dfrac{A}{\log\frac{1}{r}},\quad A>0, \quad v>0$.\\
In particular, these conditions imply
\begin{equation}\label{eq1.1}
\sup\limits_{B_{r}(x_{0})} \varPhi_{i}\big(x, \frac{v}{r}\big) \leqslant \gamma(K)\,\inf\limits_{B_{r}(x_{0})} \varPhi_{i}\big(x, \frac{v}{r}\big), \quad r\leqslant v \leqslant K, \quad i=1,2.
\end{equation}

It is well known (see \cite{BarColMing}) that the  minimizers of the corresponding  integrals of the calculus of variations satisfy Harnack's type inequality, or more generally (see \cite{HadSkrVoi}), Harnack's type inequality is valid under the conditions
$$\osc\limits_{B_{r}(x_{0})} a_{1}(x) \leqslant A \,\big[\log\frac{1}{r}\big]^{L}\,r^{q-p} , \quad \osc\limits_{B_{r}(x_{0})} a_{2}(x) \leqslant A \frac{\big[\log\log\frac{1}{r}\big]^{L}}{\log\frac{1}{r}},$$
if $L >0$ is  sufficiently small. These conditions  yield
\begin{equation}\label{eq1.2}
\sup\limits_{B_{r}(x_{0})} \varPhi_{1}\big(x, \frac{v}{r}\big) \leqslant \gamma(K)\,\inf\limits_{B_{r}(x_{0})} \varPhi_{1}\big(x, \frac{v}{r}\big), \quad r\leqslant v \leqslant K\,\lambda(r), \quad \lambda(r)=\big[\log\frac{1}{r}\big]^{-\frac{L}{q-p}},
\end{equation}
and
\begin{equation}\label{eq1.3}
\sup\limits_{B_{r}(x_{0})} \varPhi_{2}\big(x, \frac{v}{r}\big) \leqslant \gamma(K)\,\Lambda(r)\,\inf\limits_{B_{r}(x_{0})} \varPhi_{2}\big(x, \frac{v}{r}\big), \quad r\leqslant v \leqslant K,\quad \Lambda(r)=\big[\log\log\frac{1}{r}\big]^{L}.
\end{equation}

To take into account the non-uniformly elliptic case, we set
$$a(x)= \big|\log\big|\log\frac{1}{|x-x_{0}|}\big|\big|^{L_{1}},\quad x_{0}\in \Omega,$$
and let $\varPhi_{1}(v)=v^{p} +v^{q}$,\,\, $\varPhi_{2}(v)=v^{p}\big(1+ \log(1+v)\big)$, then
$$\gamma^{-1}\,a(x)\varPhi_{i}(v)\leqslant \varPhi_{i}(x, v) \leqslant \gamma\,\varPhi_{i}(v),\quad L_{1} <0,\quad i=1,2,$$
$$\gamma^{-1}\,\varPhi_{i}(v)\leqslant \varPhi_{i}(x, v) \leqslant \gamma\,a(x)\,\varPhi_{i}(v),\quad L_{1} >0,\quad i=1,2$$
provided that $B_{r}(x_{0})\subset B_{R}(x_{0})\subset \Omega$ and $R$ is sufficiently small  and  the bounded local solutions of the corresponding elliptic equations satisfy  Harnack's type inequality \cite{HadSavSkrVoi} if
\begin{equation}\label{eq1.4}
\dfrac{1}{a(x)} \in L^{t}(\Omega)\quad \text{and}\quad a(x)\in L^{s}(\Omega)
\end{equation}
with some $t$, $s >1$, i.e. if $L_{1}$ is sufficiently small.  In this paper, our aim is to combine logarithmic, non-logarithmic and non uniformly elliptic conditions \eqref{eq1.1}--\eqref{eq1.4}. Obviously, conditions \eqref{eq1.1}--\eqref{eq1.4} imply for i=1,2
\begin{equation}\label{eq1.5}
\bigg(r^{-n}\int\limits_{B_{r}(x_{0})}[\varPhi_{i}\big(x,\frac{v}{r}\big)]^{s}\,dx\bigg)^{\frac{1}{s}}\bigg(r^{-n}\int\limits_{B_{r}(x_{0})}[\varPhi_{i}\big(x,\frac{v}{r}\big)]^{-t}\,dx\bigg)^{\frac{1}{t}}\leqslant \gamma(K) \Lambda(x_{0},r),\quad r\leqslant v\leqslant K\,\lambda(r),
\end{equation}
with some $s$, $t$ and the precise choice of $\lambda(r)$ and $\Lambda(x_{0}, r)$.

Another interesting example  is the energy integral $\int\limits_{\Omega} \varPhi_{3}(x, |\nabla u|)\,dx$,
\begin{equation*}
\varPhi_{3}(x, v)= v^{p(x)}, \quad  \osc\limits_{B_{r}(x_{0})} p(x) \leqslant \frac{\bar{\mu}(r)}{\log\frac{1}{r}},\quad
\lim\limits_{r\rightarrow 0}\bar{\mu}(r)=\infty,\quad \lim\limits_{r\rightarrow 0}\frac{\bar{\mu}(r)}{\log\frac{1}{r}}=0,\quad v > 0.
\end{equation*}
It is known that the solutions of the corresponding equations and the minimizers of the corresponding integrals satisfy the Harnack type  inequality (\cite{Alhutov97}) if $\mu(r)\equiv const$, or more generally (see \cite{AlkhSurnApplAn19, AlkhSurnJmathSci20, AlkhSurnAlgAn19, SurnPrepr2018}) if $\mu(r)=L \log\log\log\frac{1}{r}$, i.e. under conditions \eqref{eq1.3}. The bounded local solutions of the corresponding elliptic equations, as well as the  minimizers of the corresponding  integrals belong to the corresponding De Giorgi's classes, i.e. for $k\in \mathbb{R}$, $\sigma$, $r\in(0, 1)$ the following inequalities hold
\begin{equation*}
\int\limits_{B_{r(1-\sigma)}(x_{0})}|\nabla(u-k)_{\pm}|^{p(x)}\,dx\leqslant \gamma \int\limits_{B_{r}(x_{0})}\bigg(\frac{u-k}{\sigma r}\bigg)^{p(x)}_{\pm}\,dx.
\end{equation*}
Set $a_{\pm}(x, k, r)=\bigg(\dfrac{M_{\pm}(k, r)}{r}\bigg)^{p(x)-p_{-}},\quad M_{\pm}(k, r)=\sup\limits_{B_{r}(x_{0})}(u-k)_{\pm}$ and $p_{-}=\min\limits_{B_{r}(x_{0})} p(x)$, then by the Young inequality
\begin{equation*}
\int\limits_{B_{r(1-\sigma)}(x_{0})}a_{\pm}(x, k, r)|\nabla(u-k)_{\pm}|^{p_{-}}\,dx \leqslant \gamma\,\sigma^{-\gamma}\bigg(\frac{M_{\pm}(k, r)}{r}\bigg)^{p_{-}}\int\limits_{B_{r}(x_{0})\cap\{(u-k)_{\pm} > 0\}}a_{\pm}(x, k, r)\,dx.
\end{equation*}
There are two possibilities how to use this inequality. The first one is almost standard, by our assumptions on the function $p(x)$ we obtain
\begin{equation*}
\int\limits_{B_{r(1-\sigma)}(x_{0})}|\nabla(u-k)_{\pm}|^{p_{-}}\,dx \leqslant \gamma\,\sigma^{-\gamma}\,\exp(\gamma\,\bar{\mu}(r))\,\bigg(\frac{M_{\pm}(k, r)}{r}\bigg)^{p_{-}}\big|B_{r}(x_{0})\cap\{(u-k)_{\pm} > 0\}\big|,
\end{equation*}
provided that $M_{\pm}(k, r) \geqslant r$. This estimate leads us to the condition (see e.g. \cite{AlkhSurnApplAn19, AlkhSurnJmathSci20, AlkhSurnAlgAn19, SurnPrepr2018})
\begin{equation}\label{eq1.6}
\int\limits_{0} \exp\big(-\gamma_{1}\,\exp(\gamma_{2}\,\bar{\mu}(r)\,\big)\frac{dr}{r}=\infty,
\end{equation}
with some $\gamma_{1}, \gamma_{2} >0$. The function $\bar{\mu}(r)=L\log\log\log\frac{1}{r}$ satisfies \eqref{eq1.6} if $L >0$ is sufficiently small.

This condition can be improved, namely, it turns out that the function $a_{\pm}(x, k, r)$ with $p(x)= p+ L\dfrac{\log\log\frac{1}{|x-x_{0}|}}{\log\frac{1}{|x-x_{0}|}}$  satisfies the condition (see \cite{SkrYev})
\begin{equation}\label{eq1.7}
\bigg(r^{-n}\,\int\limits_{B_{r}(x_{0})} [a_{\pm}(x, k, r)]^{-t}\,dx\bigg)^{\frac{1}{t}}
\bigg(r^{-n}\,\int\limits_{B_{r}(x_{0})} [a_{\pm}(x, k, r)]^{s}\,dx\bigg)^{\frac{1}{s}}\leqslant \gamma(t,s),\quad t,s>0,
\end{equation}
provided that
\begin{equation*}
 \frac{1}{\log\log\frac{1}{16r}}+ \bar{\gamma}\,L\frac{\log\log\frac{1}{16r}}{\log\frac{1}{16r}}\leqslant 1,
\quad \text{and}\quad r\leqslant M_{\pm}(k,r) \leqslant 2M= 2 \sup\limits_{\Omega}|u|,
\end{equation*}
with sufficiently large $\bar{\gamma} > 0$. This condition leads us to the standard Harnack type inequality for solutions of the corresponding $p(x)$-Laplace equation. Obviously, inequalities \eqref{eq1.7} can be generalized by conditions \eqref{eq1.5}.

In this paper we also consider the integrals of this type. And, of course, it would be interesting to unify our approach. More precisely, we will prove continuity and Harnack's inequality for functions belonging to the corresponding non uniformly elliptic De Giorgi classes $DG_{\varPhi}(B_{R}(x_{0})).$

We write  $W^{1, \varPhi}(B_{R}(x_{0}))$ for the class of functions $u\in W^{1,1}(B_{R}(x_{0}))$ with \\
$\int\limits_{B_{R}(x_{0})} \varPhi(x, |\nabla u|)dx < \infty$ and we say that a measurable function $u:B_{R}(x_{0})\rightarrow \mathbb{R}$ belongs to the elliptic class $DG^{\pm}_{\varPhi}(B_{R}(x_{0}))$ if $u\in W^{1,\varPhi}(B_{R}(x_{0})) \cap L^{\infty}(B_{R}(x_{0}))$ and there exist numbers  $c >0$, $q>1$ such that for any ball $B_{r}(x_{0})\subset B_{R}(x_{0})$, any $k \in \mathbb{R}$ and any $ \sigma\in(0,1)$  the following inequalities hold:

\begin{equation}\label{eq1.8}
\int\limits_{A^{\pm}_{k,r}}  \varPhi\big(x,|\nabla u|\big) \zeta^{q}(x) dx \leqslant c \,\int\limits_{A^{\pm}_{k,r}}  \varPhi\bigg(x,\frac{(u-k)_{\pm}}{\sigma r}\bigg) dx ,
\end{equation}
here $(u-k)_{\pm}:=\max\{\pm(u-k), 0\}$,
$A^{\pm}_{k,r}:=B_{r}(x_{0})\cap \{(u-k)_{\pm}>0\}$, $\zeta(x) \in C^{\infty}_{0}(B_{r}(x_{0})),\\ 0\leqslant \zeta(x) \leqslant 1$, $\zeta(x) =1$ in $B_{(1-\sigma)r}(x_{0})$ and $| \nabla \zeta(x) | \leqslant \dfrac{1}{\sigma r}.$\quad We also say that $u\in DG^{\pm}_{\varPhi}(\Omega)$ if $u\in DG^{\pm}_{\varPhi}(B_{R}(x_{0}))$ for any
$B_{8R}(x_{0})\subset \Omega$. We set also $DG_{\varPhi}(B_{R}(x_{0}))=DG^{-}_{\varPhi}(B_{R}(x_{0}))\cup
DG^{+}_{\varPhi}(B_{R}(x_{0}))$ and $DG_{\varPhi}(\Omega)=DG^{-}_{\varPhi}(\Omega)\cup
DG^{+}_{\varPhi}(\Omega).$

Further we suppose that $\varPhi(x, v):B_{R}(x_{0})\times \mathbb{R}_{+}\rightarrow \mathbb{R}_{+}$ is a non-negative function satisfying the following properties: for any $x\in B_{R}(x_{0})$
 the function $ v\rightarrow \varPhi(x, v)$ is increasing and\\
 $\lim\limits_{ v\rightarrow0}\varPhi(x, v)=0$,
 $\lim\limits_{ v\rightarrow +\infty}\varPhi(x, v)=+\infty$. We also assume that
\begin{itemize}
\item[($\varPhi_{0}$)]
There exists $c_{0}\geqslant 1$ such that for any $x\in B_{R}(x_{0})$ there holds
\begin{equation*}
c^{-1}_{0} \leqslant \varPhi(x, 1) \leqslant c_{0}.
\end{equation*}
\end{itemize}

 \begin{itemize}
\item[($\varPhi$)]
There exist $1<p<q$ such that
for $x\in B_{R}(x_{0})$ and for $ w\geqslant v > 0$ there holds
\begin{equation*}
\left( \frac{w}{v} \right)^{p} \leqslant\frac{\varPhi(x, w)}{\varPhi(x, v)}\leqslant
 \left( \frac{w}{v} \right)^{q}.
\end{equation*}
\end{itemize}

\begin{itemize}
\item[\big($\varPhi^{\lambda}_{\Lambda, x_{0}}$\big)]
There exist   continuous, non-decreasing function
$0< \lambda(r) \leqslant 1$  and continuous, non-increasing function $\Lambda_{\lambda}(x_{0}, r)\geqslant 1$ on the interval $(0, R)$ such that for any $B_{r}(x_{0}) \subset B_{R}(x_{0})$, for any  $K > 0$   there holds
\begin{equation*}
\sup\limits_{r\leqslant v\leqslant K\,\lambda(r)}\Lambda_{\varPhi}\big(x_{0}, r, \frac{v}{r}\big) \leqslant c_{1}(K)\Lambda_{\lambda}(x_{0}, r),
\end{equation*}

\begin{equation*}
\quad \frac{1}{tp} + \frac{1}{sp}  < \frac{1}{n},\quad t\in\big(\max(1,\frac{1}{p-1}), \infty \big], \quad s\in(1, \infty],
\end{equation*}
here $c_{1}(K)$ is some fixed positive number depending on $K$, $s$ and $t$ and
\end{itemize}
\begin{equation*}
\Lambda_{\varPhi}\big(x_{0}, r, \frac{v}{r}\big):=\Lambda_{-,\varPhi}\big(x_{0}, r, \frac{v}{r}\big)\,\,\Lambda_{+,\varPhi}\big(x_{0}, r, \frac{v}{r}\big),
\end{equation*}
\begin{equation*}
\Lambda_{-,\varPhi}\big(x_{0}, r, \frac{v}{r} \big):=\bigg(r^{-n}\int\limits_{B_{r}(x_{0})}\big[\varPhi(x, \frac{v}{r})\big]^{-t}\,dx\bigg)^{\frac{1}{t}}, \Lambda_{+,\varPhi}\big(x_{0}, r, \frac{v}{r} \big):=\bigg(r^{-n}\int\limits_{B_{r}(x_{0})}\big[\varPhi(x, \frac{v}{r})\big]^{s}\,dx\bigg)^{\frac{1}{s}}.
\end{equation*}

We will also write $(\varPhi^{\lambda}_{\Lambda})$ if condition $(\varPhi^{\lambda}_{\Lambda, x_{0}})$ holds for any $B_{R}(x_{0})\subset B_{8R}(x_{0})\subset \Omega$ and set
\begin{equation*}
\Lambda_{\lambda}(r):=\sup\limits_{x_{0}\in \Omega, B_{8R}(x_{0})\subset \Omega} \Lambda_{\lambda}(x_{0},r).
\end{equation*}
\begin{remark}\label{rem1.1}
Note that in the logarithmic case, i.e. if $\lambda(r)\equiv \Lambda_{\lambda}(r) \equiv 1$ functions from $DG_{\varPhi}(\Omega)$ belong to the standard De Giorgi class $D G_{p\frac{t}{t+1}}(\Omega)$, $p\frac{t}{t+1} >1$ (see Lemma \ref{lem2.2} below), so continuity and Harnack's inequality follow directly from results of \cite{LadUr} and \cite{DiBTr}.
\end{remark}

\begin{remark}\label{rem1.2}
We note that condition $(\varPhi^{\lambda}_{\Lambda})$ generalizes known conditions on the function $\varPhi$, for example, that is condition $(A1-n)$ from \cite{HarHasZAn19, HarHastLee18, HarHastToiv17} in the logarithmic case, i.e. if $\lambda(r)\equiv \Lambda_{1}(r)\equiv\\ \equiv const$. This condition generalizes conditions $\varPhi_{\lambda}$ and $\varPhi_{\mu}$ (see \cite{HadSkrVoi}) in the non-logarithmic case. Moreover, this condition generalizes the non-uniformly elliptic condition(see \cite{HadSavSkrVoi}). And finally, condition $\varPhi^{\lambda}_{\Lambda, x_{0}}$ includes conditions of the type \eqref{eq1.7}.
\end{remark}

Sometimes we will also need the following technical assumption
\begin{itemize}
\item[($\lambda$)]
There exist positive constants $c_{2}$ and $c_{3}$ such that
\begin{equation*}
\lambda(\rho)\leqslant \bigg(\frac{\rho}{r}\bigg)^{c_{2}}\lambda(r),\quad
\Lambda_{\lambda}(x_{0},r)\leqslant \bigg(\frac{\rho}{r}\bigg)^{c_{3}} \Lambda_{\lambda}(x_{0}, \rho),\quad 0 < r \leqslant \rho.
\end{equation*}
\end{itemize}

We refer to the parameters $n$, $p$, $q$, $t$, $s$,  $c$, $c_{0}$, $M(R):=\sup\limits_{B_{R}(x_{0})}|u|$, $c_{1}(M(R))$, $c_{2}$ and $c_{3}$  as our structural  data, and we write $\gamma$ if it can be quantitatively determined a priory in terms of the above
quantities. The generic constant $\gamma$ may change from line to line. In general, we assume that $M:=\sup\limits_{\Omega}|u|$ and $c_{1}(M)$ are also the data. Our first  result is the interior continuity of the functions belonging to the corresponding De Giorgi classes.

\begin{theorem}\label{th1.1}
 Let $u\in DG_{\varPhi}(B_{R}(x_{0}))$ and let  conditions  \,$(\varPhi_{0})$, \,$(\varPhi)$, \,$(\varPhi^{\lambda}_{\Lambda, x_{0}})$   be fulfilled. There exist numbers $C_{1}$, $\beta_{1} >0$ depending only on the data such that if
\begin{equation}\label{eq1.9}
\int\limits_{0} \exp\big(C_{1} \big[\Lambda_{\lambda}(x_{0}, r)\big]^{\beta_{1}}\big) \frac{dr}{\lambda(r)} < + \infty,\quad
\int\limits_{0} \lambda(r)\,\exp\big(-C_{1} \big[\Lambda_{\lambda}(x_{0}, r)\big]^{\beta_{1}}\big) \frac{dr}{r} = + \infty,
\end{equation}
then $u(x)$ is continuous at point $x_{0}$.

If additionally, $u\in DG_{\varPhi}(\Omega)$, condition $(\varPhi^{\lambda}_{\Lambda})$ holds and
\begin{equation}\label{eq1.10}
\int\limits_{0} \exp\big(C_{1} \big[\Lambda_{\lambda}(r)\big]^{\beta_{1}}\big) \frac{dr}{\lambda(r)} < + \infty,\quad
\int\limits_{0} \lambda(r)\,\exp\big(-C_{1} \big[\Lambda_{\lambda}(r)\big]^{\beta_{1}}\big) \frac{dr}{r} = + \infty,
\end{equation}
then $u(x) \in C(\Omega)$.
\end{theorem}
Here some typical examples of the function $\varPhi$ which satisfies the conditions of the above theorem .

$\bullet$ The function $\varPhi_{1}(x, v)= v^{p}+a(x) v^{q}$ satisfies condition $(\varPhi^{\lambda}_{1,x_{0}})$ with $\lambda(r)=[\log\frac{1}{r}]^{-L}$ and \\ $\Lambda_{\lambda}(x_{0},r)\equiv 1$ if $\osc\limits_{B_{r}(x_{0})} a(x) \leqslant A r^{q-p}\,[\log\frac{1}{r}]^{L(q-p)}$  and $a(x_{0})=0$.     Condition \eqref{eq1.9} holds if $L\leqslant 1$. If $a(x_{0})>0$, then $a(x) \asymp a(x_{0})$, provided that $R$ is small enough,  condition $(\varPhi^{1}_{1, x_{0}})$ holds with $\lambda(r)\equiv \Lambda_{1}(x_{0},r)\equiv 1$. Condition \eqref{eq1.9} is always satisfied. \\ The function $\varPhi_{1}(x, v)$ satisfies condition $(\varPhi^{1}_{\Lambda, x_{0}})$ with $\lambda(r)\equiv 1$ and $\Lambda_{1}(x_{0}, r)=[\log\log\frac{1}{r}]^{L}$, $L>0$ provided that $[\log\log\frac{1}{|x-x_{0}|}]^{-L}\leqslant a(x)\leqslant 1$. Condition \eqref{eq1.9} holds if $L\,\beta_{1} <1$.

$\bullet$ The function $\varPhi_{2}(x,v)= v^{p}\big( 1+a(x) \log(1 +v)\big)$ satisfies condition $(\varPhi^{1}_{\Lambda,x_{0}})$ with $\lambda(r)\equiv 1$ and $\Lambda_{1}(x_{0}, r)=[\log\log\frac{1}{r}]^{L}$ if $\osc\limits_{B_{r}(x_{0})} a(x) \leqslant A\dfrac{[\log\log\frac{1}{r}]^{L}}{\log\frac{1}{r}}$ and $a(x_{0})=0$, provided that $R$ is sufficiently small. Condition \eqref{eq1.9} holds if $L\beta_{1} <1$.\quad If $a(x_{0})>0$, then $a(x) \asymp a(x_{0})$, provided that $R$ is small enough,  condition $(\varPhi^{1}_{1,x_{0}})$ holds with $\lambda(r)\equiv \Lambda_{1}(x_{0},r)\equiv 1$. Condition \eqref{eq1.9} is always satisfied.\\ The function $\varPhi_{2}(x, v)$ satisfies condition $(\varPhi^{1}_{\Lambda,x_{0}})$ with $\lambda(r)\equiv 1$ and $\Lambda_{1}(x_{0}, r)=[\log\log\frac{1}{r}]^{L}$, $L>0$ provided that $[\log\log\frac{1}{|x-x_{0}|}]^{-L}\leqslant a(x)\leqslant 1$. Condition \eqref{eq1.9} holds if $L\,\beta_{1} <1$.

$\bullet$ The function $\varPhi_{3}(x, v)=v^{p(x)}$ satisfies condition $(\varPhi^{1}_{1,x_{0}})$ with $\lambda(r)\equiv 1$ and \\$\Lambda_{1}(x_{0}, r)=[\log\log\frac{1}{r}]^{L}$ if $\osc\limits_{B_{r}(x_{0})} p(x)\leqslant L\frac{\log\log\log\frac{1}{r}}{\log\frac{1}{r}}$. Condition \eqref{eq1.9} holds if $L\beta_{1} <1$.\\
The function $\varPhi_{3}(x, v)$ satisfies condition $(\varPhi^{1}_{1,x_{0}})$ with $\lambda(r)\equiv \Lambda_{1}(x_{0}, r)\equiv 1$, if \\$p(x)= p \pm L \dfrac{\log\log\frac{1}{|x-x_{0}|}}{\log\frac{1}{|x-x_{0}|}},\quad L>0$, provided that $R$ is small enough. Condition \eqref{eq1.9} is always holds.

$\bullet$ The function $\varPhi_{4}(x, v)= v^{p} \big(1+ \log(1+a(x) v)\big)$ satisfies condition $(\varPhi^{\lambda}_{1,x_{0}})$ with \\$\lambda(r)=[\log\frac{1}{r}]^{-L}$ and $\Lambda_{1}(x_{0},r)\equiv 1$ if $\osc\limits_{B_{r}(x_{0})} a(x) \leqslant A r [\log\frac{1}{r}]^{L}$ and $a(x_{0})=0$.
Condition \eqref{eq1.9} holds if $L \leqslant 1$. \quad If $a(x_{0}) >0$, then $a(x) \asymp a(x_{0})$, provided that $R$ is small enough, condition $(\varPhi^{1}_{1,x_{0}})$ holds with $\lambda(r)\equiv \Lambda_{1}(x_{0}, r)\equiv 1$. Condition \eqref{eq1.9} is always satisfied.

Next result is the Harnack inequality. We will distinguish several cases, first we will assume that $\Lambda_{\lambda}(r)\leqslant const$, $0< r
\leqslant R$. Note that the case $\lim\limits_{r\rightarrow 0} \Lambda_{1}(r) =\infty$ is possible.

\begin{theorem}\label{th1.2}
Let $u \in DG^{-}_{\varPhi}(\Omega)$, $u\geqslant 0$, let conditions $(\varPhi_{0})$, $(\varPhi)$,\,
$(\varPhi^{\lambda}_{1})$, $(\lambda)$ be fulfilled. Then there exist numbers $C_{2} > 0$, $\theta\in (0,1)$ depending only on the data  such that
\begin{equation}\label{eq1.11}
\bigg(\rho^{-n}\int\limits_{B_{\rho}(x_{0})} u^{\theta}\,dx\bigg)^{\frac{1}{\theta}} \leqslant \frac{C_{2}}{\lambda(\rho)} \big\{\,\min\limits_{B_{\frac{\rho}{2}}(x_{0})}\,u + \rho\,\big\},\quad 0< \rho \leqslant R,
\end{equation}
provided that $ B_{8R}(x_{0}) \subset \Omega$.

In addition, if  $u \in DG_{\varPhi}(\Omega)$ and condition $(\varPhi^{1}_{\Lambda})$ holds , then there exist numbers $C_{3}$, $\beta_{2} >0$
depending only on the data such that
\begin{equation}\label{eq1.12}
\max\limits_{B_{\frac{\rho}{2}}(x_{0})} u \leqslant C_{3}\,\frac{[\Lambda_{1}(\rho)]^{\beta_{2}}}{\lambda(\rho)} \big\{\,\min\limits_{B_{\frac{\rho}{2}}(x_{0})}\,u + \rho\,\big\}, \quad 0< \rho \leqslant R,
\end{equation}
provided that $B_{8R}(x_{0}) \subset \Omega$.
\end{theorem}

 We formulate our next theorem under the assumption $\lambda(r)\equiv 1$, moreover, its formulation  requires more complicated conditions on the function $\Lambda_{1}(r)$, so we will prove it only in the model case, namely, we will assume that $\Lambda_{1}(r) =[\log\log\frac{1}{r}]^{L}, L>0$.

\begin{theorem}\label{th1.3}
Let $u \in DG_{\varPhi}(\Omega)\cap C(\Omega)$, $u\geqslant 0$ and let conditions  $(\varPhi_{0})$, $(\varPhi)$,\, $(\varPhi^{1}_{\Lambda})$   be fulfilled. Let $\Lambda_{1}(\rho)=[\log\log\frac{1}{\rho}\big]^{L}$,\quad $\rho\in(0, 1)$, $L > 0$. Then there exists number $C_{4} >0$, depending only on the data and $L$ such that
\begin{equation}\label{eq1.13}
u(x_{0}) \leqslant C_{4}\,\log\frac{1}{\rho}\,\big\{\min\limits_{B_{\frac{\rho}{2}}(x_{0})} u + \,\rho\,\big\},
\quad 0< \rho \leqslant R,
\end{equation}
provided that $B_{8R}(x_{0}) \subset \Omega$ and $L$ is small enough.
\end{theorem}

We prove our most general result only for solutions of the corresponding equations. More precisely, we are concerned with elliptic equations
\begin{equation}\label{eq1.14}
div\bigg(\varPhi(x, |\nabla u|) \frac{\nabla u}{| \nabla u|^{2}}\bigg) =0, \quad x\in \Omega.
\end{equation}
We say that a function $u$ is a weak sub(super)-solution to Eq. \eqref{eq1.14} if $u \in W^{1,\varPhi}(\Omega)$ and the integral identity
\begin{equation}\label{eq1.15}
\int\limits_{\Omega}\,\varPhi(x, |\nabla u|) \frac{\nabla u}{| \nabla u|^{2}}\nabla \eta\,dx \leqslant(\geqslant)= 0,
\end{equation}
holds for all non-negative test functions $\eta \in W^{1,\varPhi}_{0}(\Omega)$.

The next result is Harnack's inequality under the point condition $(\varPhi^{\lambda}_{\Lambda, x_{0}})$.
\begin{theorem}\label{th1.4}
Let $u$ be a non-negative bounded weak super-solution to Eq. \eqref{eq1.14} and let conditions $(\varPhi_{0})$, $(\varPhi)$, $(\varPhi^{\lambda}_{\Lambda, x_{0}})$ and $(\lambda)$ be fulfilled. Assume also that
\begin{equation}\label{eq1.16}
\big( \varPhi(x, |\xi|)\frac{\xi}{|\xi|^{2}} -\varPhi(x, |\zeta|)\frac{\zeta}{|\zeta|^{2}} \big)(\xi -\zeta)  >0,\quad \xi, \zeta \in \mathbb{R}^{n},\quad \xi \ne \zeta,\quad x\in \Omega.
\end{equation}
Then there exist numbers $C_{5}$, $C_{6}> 0$, $\theta \in(0,1)$ depending only on the data such that
\begin{equation}\label{eq1.17}
\bigg(\rho^{-n}\,\int\limits_{B_{\rho}(x_{0})} u^{\theta}\,dx \bigg)^{\frac{1}{\theta}} \leqslant \frac{1}{\lambda(\rho)}\exp\big(C_{5}\big[\Lambda_{\lambda}(x_{0}, \rho)\big]^{C_{6}}\big) \big\{\,\min\limits_{B_{\frac{\rho}{2}}(x_{0})} u  + \rho \big\},
\end{equation}
provided that $B_{8\rho}(x_{0}) \subset \Omega$.

In addition, if $u$ is a  non-negative bounded  weak solution to Eq. \eqref{eq1.14}, then
\begin{equation}\label{eq1.18}
\max\limits_{B_{\frac{\rho}{2}}(x_{0})} u \leqslant \frac{1}{\lambda(\rho)}\exp\big(C_{7}\big[\Lambda_{\lambda}(x_{0}, \rho)\big]^{C_{8}}\big) \big\{\,\min\limits_{B_{\frac{\rho}{2}}(x_{0})} u  + \rho \big\},
\end{equation}
provided that $B_{8\rho}(x_{0}) \subset \Omega$. Here $C_{7}$, $C_{8} >0$ depend only on the data.
\end{theorem}

Before describing the method of proof, a few words about the history of the problem. Qualitative properties of functions belonging to the corresponding De Giorgi classes in the standard case, i.e. if $p=q$ are well known(we refer the reader to the well-known monograph of Ladyzhenskaya and Ural'tseva  \cite{LadUr} and to the seminal paper of DiBenedetto and Trudinger \cite{DiBTr}).Harnack's inequality for non uniformly elliptic equations has been known since the well-known paper of Trudinger \cite{Tru}.

The study of regularity of minima of functionals with non-standard growth has been initiated by Zhikov
\cite{ZhikIzv1983, ZhikIzv1986, ZhikJMathPh94, ZhikJMathPh9798, ZhikKozlOlein94},
Marcellini \cite{Marcellini1989, Marcellini1991}, and Lieberman \cite{Lieberman91},
and in the last thirty years, the qualitative theory
of second order elliptic equations with so-called log-condition
( if $\lambda(r)\equiv \Lambda_{1}(x_{0}, r) \equiv 1$) has been actively developed.Moreover,many authors have established local boundedness, Harnack's inequality  and continuity of solutions to such equations without or with singular lower order terms, as well as of local minimizers, $Q$-minimizers, and $\omega$-minimizers of the corresponding minimization problems(see, e.g. \cite{Alhutov97, AlhutovMathSb05, AlhutovKrash04, AlhutovKrash08, AlkSur, BarColMing, BarColMingStPt16, BarColMingCalc.Var.18, BenHarHasKarp20, BurchSkrPotAn, ColMing218, ColMing15, ColMingJFnctAn16, CupMarMas, DienHarHastRuzVarEpn, Fan1995, FanZhao1999, HarHastOrlicz, HarHasZAn19, HarHastLee18, HarHastToiv17, Krash2002, OkNA20, Skr, Sur } and references therein).

The case when conditions ($ \varPhi^{\lambda}_{\Lambda, x_{0}}$) hold differs substantially from the
logarithmic case. To our knowledge there are few results in this direction.
Zhikov \cite{ZhikPOMI04}
obtained a generalization of the logarithmic condition which guarantees the density
of smooth functions in Sobolev space $W^{1,p(x)}(\Omega)$.  This result holds if $1<p\leqslant p(x)$ and
$$
|p(x)-p(y)|\leqslant
\frac{ |\log |\log \mu(|x-y|)| |}{|\log |x-y||},
\quad x,y\in\Omega, \quad x\neq y,  \int\limits_{0} (\mu(r))^{-\frac{n}{p}} \frac{dr}{r} = + \infty,
$$
Particularly, the function $\mu(r)=(\log\frac{1}{r})^{L}$ satisfies the above condition if $L\leqslant \frac{p}{n}$.

Interior continuity, continuity up to the boundary and Harnack's inequality
to the $p(x)$-Laplace equation were proved in \cite{AlhutovKrash08, AlkhSurnAlgAn19, SurnPrepr2018}
under the condition \eqref{eq1.1}. These results were generalized in \cite{SkrVoitNA20, ShSkrVoit20}
for a wide class of elliptic  equations with non-logarithmic Orlicz growth.
Particularly, Harnack's inequality  was proved in \cite{SkrVoitNA20}  under condition \eqref{eq1.6}. In the proof, the authors used Trudinger's ideas \cite{Tru}.
Qualitative properties for solutions of non uniformly elliptic equations with non-standard growth
under the non-logarithmic conditions were considered in \cite{HadSavSkrVoi}.

As it was mentioned, in this paper we cover the non-uniformly elliptic case
and the case of variable exponent of the type \eqref{eq1.7}.

The main difficulty arising in the proof of the main results is related to the so-called
theorem on the expansion of positivity. Roughly speaking, having information on the measure
of the ''positivity set'' of $u$ over the ball $B_{r}(\bar{x})$:
$$|\{x \in B_{r}(\bar{x}) : u(x) \geqslant N \}| \geqslant \alpha(r)|B_{r}(\bar{x})|,$$
with some $r, N >0$ and  $\alpha(r) \in (0, 1)$, we cannot use the classical approach  of of Krylov and Safonov \cite{KrlvSfnv1980}, DiBenedetto and  Trudinger \cite{DiBTr} as it was done in the logarithmic case, i.e. if $\alpha$ is independent of $r$(see e.g. \cite{BarColMing}). Difficulties arise not only due to the presence of a constant $\alpha(r)$ depending on $r$, but also because in the process of iteration from $B_{r}(\bar{x})$ to $B_{\rho}(x_{0})$ an additional factor arises, which  can be estimated only under conditions of Theorems \ref{th1.2} and  \ref{th1.3}. So, first we prove the following expansion of positivity theorem.

\begin{theorem}\label{th1.5}
Let $u \in DG_{\varPhi}(\Omega) , u \geqslant 0$ , let $x_{0} \in \Omega$ be such that $B_{8R}(x_{0}) \subset \Omega$,and let conditions
$(\varPhi_{0})$, $(\varPhi)$,\,$(\varPhi^{\lambda}_{\Lambda})$, $(\lambda)$ be fulfilled ,  assume also that
\begin{equation}\label{eq1.19}
|\{B_{r}(y) : u > N \} | \geqslant \alpha |B_{r}(y)| ,
\end{equation}
with some $\alpha \in (0,1)$, some $0<N<M$ and $B_{r}(y) \subset B_{\rho}(x_{0})\subset B_{R}(x_{0})$, then there exist numbers  $\varepsilon_{0}, \gamma, c, \beta, \tau_{1}, \tau_{2} >0$ depending only on the data such that
\begin{equation}\label{eq1.20}
N\lambda(\rho)\alpha^{\tau_{1}} \leqslant \gamma\,\bigg(\frac{\rho}{r}\bigg)^{\tau_{2}}\exp\big(c\int\limits_{\bar{r}}^{\rho}[\Lambda_{\lambda}(s)]^{\beta}\frac{ds}{s}\big)\big\{\min\limits_{B_{\frac{\rho}{2}}(x_{0})} u + \rho \big\},\quad \bar{r}=\varepsilon_{0}\alpha^{2}\frac{r}{\Lambda_{\lambda}(r)}.
\end{equation}
\end{theorem}

The main step in the proof of Theorem \ref{th1.5} is the following local clustering lemma due to DiBenedetto,Gianazza and Vespri \cite{DiBGiV} (see also \cite{DiBGiaVes, Liao, Sur}).
\begin{lemma}\label{lem1.1}
Let $K_{r}(y)$ be a cube in $\mathbb{R}^{n}$ of edge $r$ centered at $y$ and let $u\in W^{1,1}(K_{r}(y))$ satisfies
\begin{equation}\label{eq1.21}
||(u-k)_{-}||_{W^{1,1}(K_{r}(y))} \leqslant \mathcal{K}\,k\,r^{n-1},\,\,\,\,\,\,and\,\,\,\,\,\, |\{K_{r}(y) : u\geqslant k \}|\geqslant \alpha |K_{r}(y)|,
\end{equation}
with some $\alpha\in (0,1)$, $k\in\mathbb{R}^{1}$ and $\mathcal{K} >0$.Then for any $\xi \in (0,1)$ and any $\nu\in (0,1)$ there exists $\bar{x} \in K_{r}(y)$ and $\varepsilon=\varepsilon(n) \in (0,1)$ such that
\begin{equation}\label{eq1.22}
|\{K_{\bar{r}}(\bar{x}):  u\geqslant \xi\,k \}| \geqslant (1-\nu) |K_{\bar{r}}(y)|,\,\,\, \bar{r}:=\varepsilon \alpha^{2}\frac{(1-\xi)\nu}{\mathcal{K}}\,r.
\end{equation}
\end{lemma}

As it was already mentioned during the iteration from $B_{r}(\bar{x})$ to $B_{\rho}(x_{0})$ an additional factor arises, that even in the case $\Lambda(x_{0}, \rho)\leqslant const$ cannot be estimated.   To overcome it and to prove Theorem \ref{th1.4} we use a workaround that goes back to Mazya \cite{Maz} and Landis \cite{Landis_uspehi1963, Landis_mngrph71}. For  the proof of the following expansion of positivity theorem we use the potential-type auxiliary solutions. We also note that by the presence of the function $\lambda(r)$ in condition $(\varPhi^{\lambda}_{\Lambda, x_{0}})$ we cannot use Moser's method, adapting the ideas of Trudinger\cite{Tru} (see, e.g. \cite{AlhutovKrash08, AlkhSurnAlgAn19, SurnPrepr2018, ShSkrVoit20}).
\begin{theorem}\label{th1.6}
Let $u$ be a   non-negative weak super-solution to Eq. \eqref{eq1.14} in $\Omega$, let conditions  $(\varPhi_{0})$, $(\varPhi)$ and  ($\varPhi^{\lambda}_{\Lambda,x_{0}}$) be fulfilled, assume also that condition \eqref{eq1.16} holds. Then there exist positive constants $\gamma $, $\beta_{3}$ and $\beta_{4}$ depending only on the data, such that for any $0 < N< M$ and any $B_{8\rho}(x_{0})\subset \Omega$ there holds
\begin{equation}\label{eq1.23}
N\,\lambda(\rho)\,\bigg(\frac{|E(\rho, N)|}{\rho^{n}}\bigg)^{\beta_{3}}  \leqslant \gamma\,\exp\big(\gamma\,\big[\Lambda_{\lambda}(x_{0},\frac{\rho}{4} )\big]^{\beta_{4}}\big)\big\{\min\limits_{B_{\frac{\rho}{2}}(x_{0})}u + \rho \big\},
\end{equation}
where $ E(\rho, N):= B_{\rho}(x_{0})\cap \{u(x) > N \}$.
\end{theorem}
To prove Theorem \ref{th1.6} we consider the solution $w$ of the following problem
\begin{equation}\label{eq1.24}
div\bigg(\varPhi(x, |\nabla w|) \frac{\nabla w}{ |\nabla w|^{2}}\bigg)= 0,\quad x\in D:=B_{8\rho}(x_{0})\setminus E,\quad w- m \psi \in W^{1,\varPhi}_{0}(D),
\end{equation}
 where $E\subset B_{\rho}(x_{0})$, $m \in (\rho, \lambda(\rho) M)$ is some fixed positive number and $\psi \in W^{1,\varPhi}_{0}(B_{8\rho}(x_{0}))$, $\psi =1 $ on $E$.

In Section \ref{Sec4} we prove upper and lower bounds for solutions of problem \eqref{eq1.24}, from which Theorem \ref{th1.6} is obtained as a simple corollary. Thanks to the use of auxiliary solutions of problem \eqref{eq1.24}, it is possible to avoid the appearance of an additional factor during the iteration from $B_{r}(x_{0})$ to $B_{\rho}(x_{0})$.

The rest of the paper contains the proof of the above theorems. In Section \ref{Sec2} we collect some auxiliary
propositions and required integral estimates of functions belonging to the corresponding De Giorgi classes. Section \ref{Sec3}
contains the proof of continuity, Theorem \ref{th1.1}, expansion of positivity, Theorem \ref{th1.5} and the proof of Harnack type inequalities,
Theorems \ref{th1.2} and \ref{th1.3}. Upper and lower bounds of auxiliary solutions are proved in Section \ref{Sec4}. A variant of the expansion of the positivity theorem, Theorem \ref{th1.6} is also proved in Section \ref{Sec4}. Finally, in Section \ref{Sec4}
we sketch a proof of  Harnack's inequality, Theorem \ref{th1.4}, leaving the details to the reader.


\section{Auxiliary material and integral estimates}\label{Sec2}

\subsection{ Auxiliary Lemma}\label{subsect2.1}

The following  lemma will be used in the sequel, it is  the well-known De Giorgi-Poincare lemma (see \cite{LadUr}, Chapter $2$).

\begin{lemma}\label{lem2.1}
{\it Let $u \in W^{1,1}(B_{r}(y))$ for some $r > 0$, and $y \in \mathbb{R}^{n}$ . Let $k, l$ be real
numbers such that $k < l$. Then there exists a constant $\gamma$ depending only on $n$ such that
\begin{equation*}
(l-k) |A^{-}_{k,r}||B_{r}(y)\setminus A^{-}_{l,r}| \leqslant \gamma r^{n+1} \int\limits_{A^{-}_{l,r}\setminus A^{-}_{k,r}} |\nabla u| dx,
\end{equation*}
where $A^{-}_{k,r} = B_{r}(y)\cap \{u < k\}$.
}
\end{lemma}

\subsection{ Local energy estimates}\label{subsect2.2}
For $\theta \in(0, p]$ and $v >0$ set $\varphi_{\theta}(x, v):=\dfrac{\varPhi(x, v)}{v^{\theta}}$.
The following lemma is a consequence of the definition of the De Giorgi class $DG_{\varPhi}(B_{R}(x_{0}))$ and  of the following analogue of the Young inequality

\begin{equation}\label{eq2.1}
\varphi_{\theta}(x,a)\,b^{\theta} \leqslant \varepsilon^{-\theta}\,\varPhi(x,a) + b^{\theta}\,\varphi_{\theta}(x,\varepsilon b),\quad \varepsilon, a, b >0,\quad \theta \in(0, p],
\end{equation}
indeed, if $b\leqslant \varepsilon^{-1} a$, then $\varphi_{\theta}(x,a)\,b^{\theta} \leqslant \varepsilon^{-\theta}\,a^{\theta}\,\varphi_{\theta}(x, a)=\varepsilon^{-\theta} \varPhi(x,a)$, and if $b \geqslant \varepsilon^{-1} a$, then since by condition $(\varPhi)$ $\varphi_{\theta}(x,\cdot)$ is non-decreasing  $\varphi_{\theta}(x,a)\,b^{\theta} \leqslant \, b^{\theta}\,\varphi_{\theta}(x,\varepsilon b)$.

\begin{lemma}\label{lem2.2}
{\it Let $u\in DG_{\varPhi}(B_{R}(x_{0}))$, then for any $r <R$ , any $k\in \mathbb{R}$ ,  any $\sigma \in (0,1)$  and any
$\theta \in [1,p\frac{t}{t+1}]$ next inequalities hold

\begin{multline}\label{eq2.2}
\int\limits_{A^{\pm}_{k,r}}
|\nabla u|^{\theta}\,\zeta^{q}(x)\, dx \leqslant \frac{\gamma}{\sigma^{\theta\frac{q}{p}}}\bigg(\frac{M_{\pm}(k,r)}{r}\bigg)^{\theta} r^{n}\,\big[\Lambda_{\varPhi}\big(x_{0}, r,\frac{M_{\pm}(k,r)}{r} \big)\big]^{\frac{\theta}{p}} \bigg(\frac{|A^{\pm}_{k,r}|}{|B_{r}(x_{0})|}\bigg)^{1-\frac{\theta}{tp}-\frac{\theta}{sp}}.
\end{multline}
Here \quad  $M_{\pm}(k,r):=\esssup\limits_{B_{r}(x_{0})} (u-k)_{\pm}$,   \quad $\zeta(x)$ is the same as in \eqref{eq1.8} and
$\Lambda_{\varPhi}\big(x_{0}, r,\frac{M_{\pm}(k,r)}{r} \big)$ was defined in $(\varPhi^{\lambda}_{\Lambda, x_{0}})$.

}
\end{lemma}

\begin{proof}
We use the H\"{o}lder inequality and inequality \eqref{eq2.1} for the function $\varphi_{p}(x, \cdot)$ with \\$ a=\dfrac{M_{\pm}(k,r)}{r}$ , $b= | \nabla u|^{p}$ and $\varepsilon=1$
\begin{multline*}
\int\limits_{A^{\pm}_{k,r}}
|\nabla u|^{p\frac{t}{t+1}}\,\zeta^{q}(x)\, dx \leqslant \\ \leqslant \bigg(\int\limits_{A^{\pm}_{k,r}}
|\nabla u|^{p}\,\varphi_{p}\bigg(x, \frac{M_{\pm}(k,r)}{r}\bigg)\,\zeta^{q}(x)\, dx \bigg)^{\frac{t}{t+1}}
\bigg(\int\limits_{A^{\pm}_{k,r}}\bigg[\varphi_{p}\bigg(x,\frac{M_{\pm}(k,r)}{r}\bigg)\bigg]^{-t}\,dx\bigg)^{\frac{1}{t+1}} \leqslant \\ \leqslant
\gamma\bigg(\frac{M_{\pm}(k,r)}{r}\bigg)^{p\frac{t}{t+1}}\bigg(\int\limits_{A^{\pm}_{k,r}}\varPhi(x, |\nabla u|)\,\zeta^{q}(x)\,dx +
\int\limits_{A^{\pm}_{k,r}}\varPhi\big(x, \frac{M_{\pm}(k,r)}{r}\big)\,\zeta^{q}(x)\,dx\bigg)^{\frac{t}{t+1}}\times \\ \times
\bigg(\int\limits_{B_{r}(x_{0})}\big[\varPhi\big(x,\frac{M_{\pm}(k,r)}{r}\big) \big]^{-t}\,dx\bigg)^{\frac{1}{t+1}}
 \leqslant  \frac{\gamma}{\sigma^{q\frac{t}{t+1}}}\bigg(\frac{M_{\pm}(k,r)}{r}\bigg)^{p\frac{t}{t+1}}\times \\ \times\bigg(\int\limits_{B_{r}(x_{0})}\big[\varPhi\big(x,\frac{M_{\pm}(k,r)}{r}\big) \big]^{-t}\,dx\bigg)^{\frac{1}{t+1}}
\bigg(\int\limits_{B_{r}(x_{0})}\big[\varPhi\big(x,\frac{M_{\pm}(k,r)}{r}\big) \big]^{s}\,dx\bigg)^{\frac{t}{s(t+1)}}
|A^{\pm}_{k,r}|^{\frac{s-1}{s}\frac{t}{t+1}},
\end{multline*}
from which by the H\"{o}lder inequality the required  \eqref{eq2.2} follows. This proves Lemma \ref{lem2.2}.
\end{proof}
In what follows, we will use only inequalities \eqref{eq2.2}, which can be taken as the definition of the corresponding
De Giorgi $DG_{\varPhi}(\Omega)$ classes.

\subsection{ A Variant of Expansion of the Positivity Lemma}\label{subsect2.3}
The following  lemma will be used in the sequel. In the proof we closely follow to \cite[Chap.\,2]{LadUr}. Let
$M(r)\geqslant \sup\limits_{B_{r}(x_{0})}u, \,m(r)\leqslant \inf\limits_{B_{r}(x_{0})}u, \,\omega(r):=M(r)-m(r)$ and
set $v_{+}(x):= M(r)-u(x)$, $v_{-}(x):= u(x)- m(r)$.

\begin{lemma}\label{lem2.3}
Let $u\in DG_{\varPhi}(B_{R}(x_{0}))$ and  let conditions  $(\varPhi_{0})$, $(\varPhi)$, $(\varPhi^{\lambda}_{\Lambda,x_{0}})$
be fulfilled. Let  $\xi  \in(0,1)$ and
assume that with some $\alpha_{0}\in(0,1)$ there holds
\begin{equation}\label{eq2.3}
\left| \left\{ x\in B_{3r/4}(x_{0}):v_{\pm}(x)
\leqslant \xi\,\omega(r) \right\} \right|
\leqslant(1-\alpha_{0})\, |B_{3r/4}(x_{0})|.
\end{equation}
Then for any $\nu\in(0,1)$ there exists number $C_{\ast}\geqslant 1$ depending only on the known data, $\alpha_{0}$, $\xi$ and  $\nu$  such that either
\begin{equation}\label{eq2.4}
\omega(r) \leqslant \,\frac{r}{\lambda(r)}\, \exp\big(C_{\ast} [\Lambda_{\lambda}\big(x_{0}, r\big)]^{\bar{\beta}_{1}}\big),
\end{equation}
or
\begin{equation}\label{eq2.5}
|\{B_{3/4 r}(x_{0}) :v_{\pm}(x)\leqslant \omega(r)\,\lambda(r)\, \exp\big(-C_{\ast} [\Lambda_{\lambda}\big(x_{0}, r\big)]^{\bar{\beta}_{1}}\big)\}| \leqslant \nu |B_{3/4 r}(x_{0})|,
\end{equation}
here $\bar{\beta}_{1} $ is some fixed positive number depending only on the data.
\end{lemma}
\begin{proof}
We provide the proof of \eqref{eq2.5} for $v_{+}$,
while the proof for $v_{-}$ is completely similar.
We set $k_{j}:= M(r)-\dfrac{\lambda(r)}{2^{j}}\omega(r)$,
$j=[\log 1/\xi]+1,2, \ldots,j_{\ast}$, where $j_{\ast}$ to be chosen. We will assume  for all $j \in[[\log 1/\xi]+1, j_{\ast}]$ that
$M_{+}(k_{j}, 3/4 r) \geqslant \dfrac{\lambda(r)}{2^{j+1}}\,\omega(r)$, because if for some $j$ this inequality is violated then the required \eqref{eq2.5} with $C_{\ast}\geqslant j_{\ast}+1$ is evident. If \eqref{eq2.4} is violated, then $M_{+}(k_{j},  r)\geqslant M_{+}(k_{j}, 3/4 r) \geqslant r$ and  since
$2^{-j-1}\,\omega(r)\,\lambda(r)\leqslant M_{+}(k_{j},r)\leqslant 2^{-j}\,\omega(r)\,\lambda(r)\leqslant \\ \leqslant 2M(R)\lambda(r)$,  by  $(\varPhi^{\lambda}_{\Lambda, x_{0}})$ we obtain that
$$
\Lambda_{\varPhi}\big(x_{0}, r,\frac{M_{+}(k_{j},r)}{r}\big)\leqslant 2^{q} \Lambda_{\varPhi}\big(x_{0}, r, \frac{\lambda(r)}{r 2^{j}}\omega(r)\big)\leqslant \gamma \Lambda_{\lambda}(x_{0}, r).
$$
Therefore , if \eqref{eq2.4} is violated, inequality \eqref{eq2.2} with $\theta=p\frac{t}{t+1}$
can be rewritten as
\begin{equation*}
\int\limits_{A^{+}_{k_{j},r}}
|\nabla u|^{\theta}\,\zeta^{\,q}\,dx \leqslant\gamma\,\bigg(\frac{\lambda(r)}{2^{j}r}\omega(r)\bigg)^{\theta}\,r^{n}\,
[\Lambda_{\lambda}\big(x_{0}, r\big)]^{\frac{\theta}{p}}\left( \frac{|A^{+}_{k_{j},r}|}{|B_{r}(x_{0})|} \right)^{\frac{1}{\kappa_{1}}}
\end{equation*}
where $\dfrac{1}{\kappa_{1}} =1- \dfrac{\theta}{sp}-\dfrac{\theta}{tp}$,\quad $\theta= p\frac{t}{t+1}$ and $\zeta\in C_{0}^{\infty}(B_{r}(x_{0}))$, $0\leqslant\zeta\leqslant1$, $\zeta=1$ in $B_{3r/4}(x_{0})$, $|\nabla\zeta|\leqslant 4/r$.
From this by Lemma \ref{lem2.1} we obtain
\begin{multline*}
\frac{\lambda(r)}{2^{j+1}}\omega(r)|A^{+}_{k_{j},3/4 r}|\leqslant  \frac{\gamma}{\alpha_{0}}\,r\int\limits_{A^{+}_{k_{j},r}\setminus A^{+}_{k_{j+1},r}} |\nabla u|\,\zeta^{\,q}\,dx\leqslant\\ \leqslant \frac{\gamma}{\alpha_{0}}\,r\bigg(\int\limits_{A^{+}_{k_{j},r}} |\nabla u|^{\theta}\,\zeta^{\,q}\,dx\bigg)^{\frac{1}{\theta}} |A^{+}_{k_{j},r}\setminus A^{+}_{k_{j+1},r}|^{1-\frac{1}{\theta}}\leqslant
\\ \leqslant \frac{\gamma}{\alpha_{0}}\,\frac{\lambda(r)}{\,2^{j}}\omega(r)\,[\Lambda_{\lambda}\big(x_{0}, r\big)]^{\frac{1}{p}}
\left( \dfrac{|A^{+}_{k_{j},r}\setminus A^{+}_{k_{j+1},r}|}
{|B_{r}(x_{0})|} \right)^{1-\frac{1}{\theta}}\left( \frac{|A^{+}_{k_{j},r}|}{|B_{r}(x_{0})|} \right)^{\frac{1}{\theta\kappa_{1}}}|B_{r}(x_{0})|,
\end{multline*}
raising the left and right hand-sides to the power $\dfrac{\theta}{\theta-1}$ and summing up the resulting inequalities in $j$, $j=[\log 1/\xi]+1,2, \ldots,j_{\ast}$ ,we conclude that
$$
(j_{\ast} -[\log1/ \xi]-1)\bigg(\frac{|A^{+}_{k_{j_{\ast}},3/4 r}|}{B_{3/4 r}(x_{0})|}\bigg)^{\frac{\theta}{\theta-1}}\leqslant \gamma \alpha_{0}^{-\frac{\theta}{\theta-1}}\,[\Lambda_{\lambda}\big(x_{0}, r\big)]^{\frac{\theta}{p(\theta-1)}}.
$$
Choosing $j_{\ast}$ by the condition
$$j_{\ast} -[\log1/ \xi]-1 \geqslant \gamma\,\nu^{-\frac{\theta}{\theta-1}}\,\alpha_{0}^{-\frac{\theta}{\theta-1} }
[\Lambda_{\lambda}\big(x_{0}, r\big)]^{\frac{\theta}{p(\theta-1)}},\quad \theta=p\frac{t}{t+1},  $$
we obtain inequality \eqref{eq2.5}, which proves Lemma \ref{lem2.3} with
 $$C_{\ast} \geqslant 1+\log1/\xi +\gamma\,\alpha_{0}^{-\frac{pt}{(p-1)t-1}}\,\nu^{-\frac{pt}{(p-1)t-1}}\quad
\text{and}\quad \bar{\beta}_{1}= \frac{t}{(p-1)t-1}.$$
\end{proof}

\subsection{De Giorgi Type Lemma}\label{subsec2.4}
The following theorem is the De Giorgi type Lemma, the proof is almost standard (see e.g. \cite{LadUr}).
\begin{lemma}\label{lem2.4} Let $u \in DG_{\varPhi}(B_{R}(x_{0}))$ and let conditions  $(\varPhi_{0}), (\varPhi)$, $(\varPhi^{\lambda}_{\Lambda, x_{0}})$  be fulfilled. Fix $\xi, \eta \in (0,1)$,  there exists number  $\nu_{1} \in(0, 1)$ depending only on the data and  $\eta,$  such that if
\begin{equation}\label{eq2.6}
|\{x\in B_{r}(x_{0}) : v_{\pm}(x) \leqslant \xi\,\omega(r) \}|\leqslant \nu_{1}\,[\Lambda_{\lambda}(x_{0}, r)]^{-\bar{\beta}_{2}}|B_{r}(x_{0})|,
\end{equation}
then either
\begin{equation}\label{eq2.7}
\eta\,(1-\eta)\,\xi\,\omega(r)\leqslant   \frac{r}{\lambda(r)},
\end{equation}
or
\begin{equation}\label{eq2.8}
v_{\pm}(x) \geqslant (1-\eta)^{2}\,\xi\,\lambda(r)\,\omega(r),\quad x \in B_{\frac{r}{2}}(x_{0}).
\end{equation}
Here $\bar{\beta}_{2}$ is some fixed positive number, depending only on the data.
\end{lemma}
\begin{proof}
For $j=0,1,2,\ldots$ we set $r_{j}:=\dfrac{r}{2}(1+2^{-j}),\,\, \bar{r}_{j}=\dfrac{r_{j}+r_{j+1}}{2},\,\,\\
k_{j}:= M(r)- \eta\,(2-\eta)\xi\,\lambda(r)\,\omega(r)
-\dfrac{(1-\eta)^{2}}{2^{j}}\xi\,\lambda(r)\,\omega(r)$, and let $\zeta_{j} \in C^{\infty}_{0}(B_{\bar{r}_{j}}(x_{0})),\,\,\\ 0 \leqslant \zeta_{j} \leqslant 1,\,\, \zeta_{j} =1$ in $B_{r_{j+1}}(x_{0}),$ and $| \nabla \zeta_{j} | \leqslant \gamma \dfrac{2^{j}}{r}.$ We assume that\\
$M_{+}(k_{\infty},r/2)\geqslant \eta (1-\eta)\,\xi\,\lambda(r)\,\omega(r)$,
because in the opposite case, the required \eqref{eq2.8} is evident.
If \eqref{eq2.7} is violated then $M_{+}(k_{\infty},r/2)\geqslant  r$.
In addition, since $\eta (1-\eta) \xi\,\lambda(r)\,\omega(r) \leqslant M_{+}(k_{j},r)\leqslant\\ \leqslant \xi\,\lambda(r)\,\omega(r)$
for $j=0,1,2,\ldots$, then by $(\varPhi^{\lambda}_{\Lambda, x_{0}})$
we obtain that
$$
\Lambda\big(x_{0}, r_{j},\frac{M_{+}(k_{j},r)}{r_{j}}\big)\leqslant \eta^{-q}\,(1-\eta)^{-q} \Lambda\big(x_{0}, r_{j}, \xi\,\frac{\lambda(r)}{r_{j}}\,\omega(r)\big) \leqslant \gamma \eta^{-q}\,(1-\eta)^{-q} \Lambda_{\lambda}(x_{0},r ).
$$
Therefore inequality \eqref{eq2.2} with $\theta=1$ can be rewritten as
$$
\int\limits_{A^{+}_{k_{j},r_{j}}}
|\nabla u|\,\zeta^{q}_{j}\,dx\leqslant \gamma\,2^{j\gamma}\,\eta^{-\frac{q}{p}}\,(1-\eta)^{-\frac{q}{p}}\,
\xi\,\lambda(r)\,\omega(r) r^{n}[\Lambda_{\lambda}(x_{0},r)]^{\frac{1}{p}}\,\left(\frac{|A^{+}_{k_{j},r_{j}}|}{|B_{r}(x_{0})|}\right)^{\frac{1}{\kappa_{2}}},
\frac{1}{\kappa_{2}}=1-\frac{1}{sp}-\frac{1}{tp}.
$$

From this, by the Sobolev embedding theorem  we obtain
$$
y_{j+1}=\frac{|A^{+}_{k_{j+1},r_{j+1}}|}{|B_{r}(x_{0})|} \leqslant \gamma\,2^{j\gamma}\,\eta^{-\frac{q}{p}}(1-\eta)^{-2-\frac{q}{p}}\,
[\Lambda_{\lambda}(x_{0}, r)]^{\frac{1}{p}} y_{j}^{1+\kappa}, \,\,j=0,1,2,... ,
$$
where $\kappa= \dfrac{1}{\kappa_{2}} - 1 +\dfrac{1}{n} =\dfrac{1}{n}-\dfrac{1}{sp}-\dfrac{1}{tp} >  0$. Choosing $\nu_{1}$, $\bar{\beta}_{2}$ from the condition
$$\nu_{1}=\gamma^{-1}\,\eta^{\frac{q}{p\kappa}}\,(1-\eta)^{\frac{2+\frac{q}{p}}{\kappa}}, \quad \bar{\beta}_{2}=\frac{1}{p\kappa},$$
and iterating the previous inequality  we arrive at the required \eqref{eq2.8},
which completes the proof of the lemma.
\end{proof}

.

\section{  Continuity and Harnack's Type Inequality , Proof of \\Theorems \ref{th1.1}--\ref{th1.3} and \ref{th1.5}}\label{Sec3}
\subsection {Continuity }\label{subsect3.1}
 Let $r, \rho$ be arbitrary such that $0 < r < \rho < R$, where $R$ is small enough. We assume that the following two alternative
cases are possible:
$$
\left|\left\{x\in B_{\frac{3}{4}r}(x_{0}):
u(x)\geqslant M(r) -\frac{1}{2}\,\omega(r) \right\} \right|
\leqslant \frac{1}{2}\,|B_{\frac{3}{4}r}(x_{0})|
$$
or
$$
\left|\left\{x\in B_{\frac{3}{4}r}(x_{0}):
u(x)\leqslant m(r) +\frac{1}{2}\,\omega(r)\right\} \right|
\leqslant \frac{1}{2}\,|B_{\frac{3}{4}r}(x_{0})|.
$$
Assume, for example, the first one. Then  Lemmas \ref{lem2.3}, \ref{lem2.4},  the choice of constant $C_{\ast}$ in Lemma
\ref{lem2.3} and  the choice of $\nu$ in Lemma \ref{lem2.4} ensure the existence of $\bar{\beta}_{3}= \bar{\beta}_{1}+ \bar{\beta}_{2}
\dfrac{pt}{(p-1)t- 1}$ such that
\begin{equation*}
\omega\big(\frac{r}{2}\big) \leqslant \bigg(1- \frac{\lambda(r)}{4}\exp\big(- \gamma [\Lambda_{\lambda}(x_{0},r)]^{\bar{\beta}_3}\bigg)\,\omega(r) +
\gamma\,\frac{r}{\lambda(r)}\,\exp\big(\gamma[\Lambda_{\lambda}(x_{0},r)]^{\bar{\beta}_3}\big),
\end{equation*}
 where  $\bar{\beta}_{1}$, $\bar{\beta}_{2}$  were defined in  Lemmas \ref{lem2.3},  \ref{lem2.4} .

Iterating this inequality, we obtain
\begin{equation*}
\omega(r) \leqslant 2\,M(R)\,\exp\bigg(- \gamma \int\limits_{r}^{2\rho}\lambda(s)\,\exp\big(-\gamma [\Lambda_{\lambda}(x_{0},s)]^{\bar{\beta}_{3}}\big)\,\frac{ds}{s}\bigg) +  \int\limits_{\frac{r}{2}}^{\rho}\exp\big(\gamma [\Lambda_{\lambda}(x_{0},s)]^{\bar{\beta}_{3}}\big)\,\frac{ds}{\lambda(s)},
\end{equation*}
from which the required continuity follows.
This completes the proof of Theorem \ref{th1.1}.

\subsection{ Proof of Theorem \ref{th1.5}}\label{subsect3.2}

Let $B_{r}(y) \subset B_{\rho}(x_{0}) \subset B_{8\rho}(x_{0}) \subset \Omega$ and let the following inequality holds
\begin{equation}\label{eq3.1}
|\{\, B_{r}(y) : u \geqslant N \,\}| \geqslant \alpha |B_{r}(y)|,
\end{equation}
for some $0<N<M$, $\alpha \in(0,1)$.

Let $\bar{r} < r$  be some number which we will fix later, we will assume that $ \sup\limits_{B_{r}(y)}(u- N\lambda(\bar{r}))_{-}\geqslant\\ \geqslant \frac{N}{2}\lambda(\bar{r})$, because in the opposite case the required inequality \eqref{eq1.20} is obvious. If $N \geqslant r$, then by  $(\varPhi^{\lambda}_{\Lambda})$, using the fact that $\lambda(\bar{r})\leqslant \lambda(r)$ if $\bar{r}\leqslant r$, we have with any $0< \bar{r} \leqslant r$
$$
\Lambda_{\varPhi}\big(y,r,\frac{M_{-}(N\lambda(\bar{r}), r)}{r}\big)\leqslant 2^{q}\Lambda_{\varPhi}\big(y,r,\frac{\lambda(\bar{r})}{r} N\big)\leqslant 2^{q}\Lambda_{\lambda}(y,r) \leqslant 2^{q}\Lambda_{\lambda}(r),
$$
apply inequality \eqref{eq2.2} with $\theta=1$ for $(u- \lambda(\bar{r})\,N)_{-}$ over the pair of balls $B_{r}(y)$ and $B_{2r}(y)$ and with arbitrary $0 <\bar{r} \leqslant r$ to obtain
$$
\int\limits_{B_{r}(y)} |\nabla (u- \lambda(\bar{r})\,N)_{-}|\,dx \leqslant \gamma \,[\Lambda_{\lambda}(r)]^{\frac{1}{p}}\,\lambda(\bar{r})\,N \,r^{n-1}.
$$
The local clustering Lemma \ref{lem1.1} with $k=\lambda(\bar{r})\,N$, $\nu=\frac{1}{4}$, $\xi=\frac{1}{2}$, $\mathcal{K}=\gamma[\Lambda_{\lambda}(r)]^{\frac{1}{p}}$  implies the existence of a point $\bar{x}\in B_{r}(y)$ and $\varepsilon \in(0,1)$ depending only on the data such that
\begin{equation*}
|\{B_{\bar{r}}(\bar{x}) : u >\frac{\lambda(\bar{r})}{2} N \}| \geqslant \frac{1}{4} |B_{\bar{r}}(\bar{x})|,\quad \bar{r}=\frac{\varepsilon\alpha^{2}}{8\gamma}\,\frac{r}{[\Lambda_{\lambda}(r)]^{\frac{1}{p}}}.
\end{equation*}
Set $\varepsilon_{0}:=\dfrac{\varepsilon}{8\gamma}$ , then the previous inequality can be rewritten as
\begin{equation}\label{eq3.2}
|\{B_{\bar{r}}(\bar{x}) : u >\frac{\lambda(\bar{r})}{2} N \}| \geqslant \frac{1}{4} |B_{\bar{r}}(\bar{x})|.
\end{equation}
From this by Lemmas \ref{lem2.3} and \ref{lem2.4} we obtain
$$
u(x)\geqslant \frac{\lambda(\bar{r})}{2}\,N \exp\big(-\gamma [\Lambda_{\lambda}(2\bar{r}))]^{\bar{\beta}_{3}}\big), \quad x \in B_{2\bar{r}}(\bar{x}),
\quad \bar{\beta}_{3}=\bar{\beta}_{1}+ \bar{\beta}_{2}\frac{pt}{(p-1)t-1},
$$
provided that
$$N \geqslant \frac{2\bar{r}}{\lambda(\bar{r})}\,\exp\big( \gamma [\Lambda_{\lambda}(2\bar{r})]^{\bar{\beta}_{3}}\big),$$
where  $\bar{\beta}_{1}$,  $\bar{\beta}_{2}$ are the constants defined in Lemmas \ref{lem2.3}, \ref{lem2.4}.

Repeating this procedure $j$-times  we obtain
$$ u(x)\geqslant 2^{-j} \,\lambda(\bar{r})\,N\, \exp\big(-\gamma \int\limits_{2\bar{r}}^{2^{j+1}\bar{r}}[\Lambda_{\lambda}(s)]^{\bar{\beta}_{3}}\frac{ds}{s}\big),\,\, x \in B_{2^{j}\bar{r}}(\bar{x}),$$
provided that
$$N \geqslant \frac{2^{j}\bar{r}}{\lambda(\bar{r})} \, \exp\big(\gamma \int\limits_{2\bar{r}}^{2^{j+1}\bar{r}}[\Lambda_{\lambda}(s)]^{\bar{\beta}_{3}}\frac{ds}{s}\big).$$
Choosing $j$ from the condition $2^{j}\bar{r}=\rho$ and using condition $(\lambda)$, from the previous we obtain
\begin{multline*}
u(x)\geqslant \gamma^{-1}\,\frac{\lambda(\bar{r})}{\Lambda_{\lambda}(r)}\,N\,\alpha^{2} \frac{r}{\rho}\exp\big(-\gamma \int\limits_{2\bar{r}}^{2\rho}[\Lambda_{\lambda}(s)]^{\bar{\beta}_{3}}\frac{ds}{s}\big)\geqslant \\ \geqslant
\gamma^{-1}\,\frac{\lambda(\rho)}{\Lambda_{\lambda}(\rho)}\,N\,\alpha^{\tau_{1}} \bigg(\frac{r}{\rho}\bigg)^{\tau_{2}}\exp\big(-\gamma \int\limits_{2\bar{r}}^{2\rho}[\Lambda_{\lambda}(s)]^{\bar{\beta}_{3}}\frac{ds}{s}\big)\geqslant\\
\geqslant\gamma^{-1}\,\lambda(\rho)\,N\,\alpha^{\tau_{1}} \bigg(\frac{r}{\rho}\bigg)^{\tau_{2}}\exp\big(-\gamma \int\limits_{2\bar{r}}^{2\rho}[\Lambda_{\lambda}(s)]^{\bar{\beta}_{3}}\frac{ds}{s}\big) ,\,\, x \in B_{\frac{\rho}{2}}(x_{0}),
\end{multline*}
provided that
\begin{equation*}
N\geqslant \gamma\,\frac{\rho}{\lambda(\rho)}\alpha^{-\tau_{1}} \bigg(\frac{\rho}{r}\bigg)^{\tau_{2}}\exp\big(\gamma \int\limits_{2\bar{r}}^{2\rho}[\Lambda_{\lambda}(s)]^{\bar{\beta}_{3}}\frac{ds}{s}\big),
\end{equation*}
which completes the proof of the theorem.

\subsection{ Proof of Theorem \ref{th1.2} } \label{subsect3.3}

Inequality \eqref{eq1.11} of Theorem \ref{th1.2} follows immediately from Theorem \ref{th1.5} with $\Lambda_{\lambda}(\rho)\leqslant \gamma$ , indeed set
$$\bar{m}(\rho)=\frac{1}{\lambda(\rho)} \{\min\limits_{B_{\frac{\rho}{2}}(x_{0})}u(x) + \rho \}.$$
 By Theorem \ref{th1.5}  with $r=\rho$ we obtain for
$\theta\in\big(0, \frac{1}{\tau_{1}}\big)$
\begin{multline}\label{eq3.3}
\rho^{-n}\,\int\limits_{B_{\rho}(x_{0})}u^{\theta}\,dx=\theta \,\rho^{-n}\,\int\limits_{0}^{\infty}|\{B_{\rho}(x_{0}): u(x) >N\}|\,N^{\theta-1}\,dN \leqslant [\bar{m}(\rho)]^{\theta}+\\+\gamma [\bar{m}(\rho)]^{\frac{1}{\tau_{1}}}\int\limits_{\bar{m}(\rho)}^{\infty} N^{\theta-\frac{1}{\tau_{1}}-1} \,dN \leqslant \frac{\gamma \tau_{1}}{1-\theta \tau_{1}}\,[\bar{m}(\rho)]^{\theta},
\end{multline}
which proves inequality \eqref{eq1.11}.

To prove \eqref{eq1.12} fix $\sigma\in(0,\frac{1}{8})$, $s\in (\frac{3}{4}\rho, \frac{7}{8}\rho)$ and  let $M_{0}:=\sup\limits_{B_{s}(x_{0})} u$, $M_{\sigma}:= \sup\limits_{B_{s(1-\sigma)}(x_{0})} u$. Fix $\bar{x}\in B_{s(1-\sigma)}(x_{0})$ and for $j= 0,1,2,....$ set $\rho_{j}:=s\frac{\sigma}{2}(1+  2^{-j})$, $B_{j}:=B_{\rho_{j}}(\bar{x})$, $k_{j}= k(1-2^{-j})$, where $k >0$    and suppose that $(u(\bar{x})-k)_{+} \geqslant \rho$, then $\sup\limits_{B_{j}}(u-k_{j})_{+} \geqslant \rho$. We will use inequality \eqref{eq2.2}, by condition $(\varPhi^{1}_{\Lambda})$ with $K=M$ we have
\begin{equation*}
\Lambda_{\varPhi}\big(\bar{x}, \rho_{j}, \sup\limits_{B_{j}}(u-k_{j})_{+}/\rho_{j}\big)\leqslant \gamma(M) \Lambda_{1}(\bar{x}, \rho_{j})\leqslant \gamma \Lambda_{1}(\rho).
\end{equation*}
Hence, inequality \eqref{eq2.2} can be rewritten as
\begin{equation*}
\int\limits_{B_{j+1}\cap\{u>k_{j}\}}|\nabla u|\,dx \leqslant \frac{\gamma}{\sigma^{\gamma}} 2^{j\gamma}\, M_{0}\,\rho^{n-1}
[\Lambda_{1}(\rho)]^{\frac{1}{p}}\bigg(\frac{|B_{j}\cap\{u>k_{j}\}|}{|B_{j}|}\bigg)^{1-\frac{1}{tp} -\frac{1}{sp}}.
\end{equation*}
Since $\bar{x}$ is an arbitrary point in $B_{s(1-\sigma)}(x_{0})$ this inequality by standard arguments ( see e.g. \cite{LadUr}) yields
\begin{equation*}
M_{\sigma}^{1+\frac{1}{\kappa}} \leqslant \gamma\,\sigma^{-\gamma}\,M_{0}^{\frac{1}{\kappa}}\,
[\Lambda_{1}(\rho)]^{\frac{1}{p\kappa}}\,\rho^{-n}\int\limits_{B_{0}}\,u\,dx +\gamma\,\rho^{1+\frac{1}{\kappa}},
\quad \kappa=\frac{1}{n}-\frac{1}{tp}-\frac{1}{sp} >0,
\end{equation*}
which by the Young inequality implies for any $\varepsilon, \theta \in(0,1)$ that
\begin{equation*}
M_{\sigma} \leqslant \varepsilon\,M_{0} +\gamma\,\sigma^{-\gamma}[\Lambda_{1}(\rho)]^{\frac{1}{p\kappa\theta}}\bigg(\rho^{-n}\int\limits_{B_{0}}\,u^{\theta}\,dx\bigg)^{\frac{1}{\theta}} +\gamma\,\rho.
\end{equation*}
Iterating this inequality we arrive at
\begin{equation}\label{eq3.4}
\sup\limits_{B_{\frac{\rho}{2}}(x_{0})} u \leqslant \gamma\,[\Lambda_{1}(\rho)]^{\frac{1}{p\kappa\theta}}\bigg(\rho^{-n}\int\limits_{B_{0}}\,u^{\theta}\,dx\bigg)^{\frac{1}{\theta}} +\gamma\,\rho.
\end{equation}
Collecting  \eqref{eq1.11} and \eqref{eq3.4} with $\theta=\dfrac{1}{2\tau_{1}}$ we arrive at the required
\eqref{eq1.12} with $\beta_{2}=(2p\kappa \tau_{1})^{-1}$, which completes the proof of Theorem \ref{th1.2}.

\subsection{ Proof of Theorem \ref{th1.3} }\label{subsect3.4}

Construct the ball $B_{\rho\tau}(x_{0})$, $\tau\in(0,1)$ and set $u_{0}:=u(x_{0})$, $M_{\tau}:=\max\limits_{B_{\rho\tau}(x_{0})}u$,\\
$N_{\tau}:=\dfrac{u_{0}}{2}\,(1-\tau)^{-l_{1}}\,\exp\bigg(l_{2}\big[\log\log\frac{1}{(1-\tau)\rho}\big]^{Ll_{3}} - l_{2}\big[\log\log\frac{1}{\rho}\big]^{Ll_{3}}\bigg)\,\exp\big(c\int\limits_{\psi((1-\tau)\rho)}^{\psi(\rho)}[\log\log\frac{1}{s}]^{L\beta}\frac{ds}{s}\big),$ where $c$ and $\beta$ the numbers defined in Theorem \ref{th1.5}, $l_{1}, l_{2}, l_{3} >0$ will be chosen  depending only on the known data and
$$\psi(r)=\varepsilon_{0}\,\alpha^{2}(r)\,r\,[\log\log\frac{1}{r}]^{-L},$$
here $\varepsilon_{0} \in(0, 1)$ is the number, defined in Theorem \ref{th1.5} and $\alpha(r) \in (0,1)$ is continuous non-increasing function, which will be defined later. Consider the equation
\begin{equation}\label{eq3.5}
M_{\tau}=N_{\tau}.
\end{equation}

Further we will assume that
\begin{equation}\label{eq3.6}
u_{0}\geqslant l_{4}\,\rho\,\,\log\frac{1}{\rho},
\end{equation}
with some $l_{4}>0$ to be fixed later.

Let $\tau_{0}\in (0,1)$ be the maximal root of equation \eqref{eq3.5} and fix $y$ by the condition \\$u(y)=\max\limits_{B_{\rho\tau_{0}}(x_{0})}u $.
Since $B_{\frac{\rho}{2}(1-\tau_{0})}(y)\subset
B_{\frac{\rho}{2}(1+\tau_{0})}(x_{0})$,
setting $r=\frac{\rho}{2}(1-\tau_{0})$ we have
\begin{multline*}
\max_{B_{r}(y)}u
\leqslant u(y) 2^{l_{1}}\,\exp\bigg(l_{2}(\big[\log\log\frac{1}{r}\big]^{L l_{3}}-\big[\log\log\frac{1}{2r}\big]^{L l_{3}})\bigg)
\,\exp\big(c\int\limits_{\psi(r)}^{\psi(2 r)}[\log\log\frac{1}{s}]^{L \beta}\frac{ds}{s}\big) \leqslant \\ \leqslant
u(y) 2^{l_{1}}\,\exp\bigg(l_{2}(\big[\log\log\frac{1}{r}\big]^{L l_{3}}-\big[\log\log\frac{1}{2r}\big]^{L l_{3}})\bigg)
\exp\big(c[\log\log\frac{1}{\psi(r)}]^{L\beta}\,\log\frac{\psi(2r)}{\psi(r)}\big).
\end{multline*}
Let us estimate the terms on the right-hand side of the previous inequality. If
\begin{equation}\label{eq3.7}
L l_{3} < 1 \quad \text{and}\quad 2 l_{2}\big[\log\frac{1}{2\rho}\big]^{-1} \leqslant 1,
\end{equation}
then
\begin{equation*}
 l_{2}(\big[\log\log\frac{1}{r}\big]^{L l_{3}}-\big[\log\log\frac{1}{2r}\big]^{L l_{3}}) \leqslant 1.
\end{equation*}
 By our assumption $\alpha(r)$ is non-increasing, therefore choosing $\rho$ sufficiently small we have
$$\frac{\psi(2r)}{\psi(r)}=2 \frac{\alpha^{2}(2r)}{\alpha^{2}(r)}\,\frac{[\log\log\frac{1}{r}]^{L}}{[\log\log\frac{1}{2r}]^{L}}
\leqslant 2\,\frac{[\log\log\frac{1}{r}]^{L}}{[\log\log\frac{1}{2r}]^{L}} \leqslant 4.$$
Moreover
$$\log\log\frac{1}{\psi(r)}=\log\log\frac{[\log\log\frac{1}{r}]^{L}}{\varepsilon_{0}r\alpha(r)}\leqslant 2\log\log\frac{1}{r},$$
provided that $\rho$ is small enough and
\begin{equation}\label{eq3.8}
 2 \log\frac{1}{\alpha(r)} \leqslant \log\frac{1}{r}.
\end{equation}
Therefore
$$\max_{B_{r}(y)}u \leqslant u(y)\,2^{l_{1}+1}\,\exp\big(8^{L \beta}\,c\,[\log\log\frac{1}{r}]^{L \beta}\big). $$

\textsl{Claim.}
There exists a positive number $\nu_{1}\in(0,1)$ depending only on the known data
 and $l_{1}$ \,\,\,such that
\begin{equation*}
\left| \left\{ x\in B_{r}(y):
u(x)\geqslant \frac{u(y)}{4}  \right\}  \right|\geqslant \nu_{1}\,[\log\log\frac{1}{r}]^{-\bar{\beta}_{2}}\,\exp\bigg(-\frac{q}{p\kappa} 8^{L\beta}\,c\,[\log\log\frac{1}{r}]^{L \beta}\bigg)\,| B_{r}(y)|,
\end{equation*}
where $\bar{\beta}_{2}$ is the number defined in Lemma \ref{lem2.4} and $\kappa= \frac{1}{n}-\frac{1}{sp}-\frac{1}{tp} >0$.

Indeed, in the opposite case we apply Lemma \ref{lem2.4} with the choices
\vskip0.3cm
$M(r)=u(y)\,2^{l_{1}+1}\,\exp\big(8^{L \beta}\,c\,[\log\log\frac{1}{r}]^{L \beta}\big) , \quad
\xi\,\omega(r) =(M(r)-\frac{1}{4}u(y)),$
\vskip0.3cm
$1-\eta=\bigg(\dfrac{M(r)-\frac{11}{16}u(y)}{M(r)-\frac{1}{4}u(y)}\bigg)^{\frac{1}{2}} < 1$ and $\eta \geqslant \dfrac{1}{\gamma(l_{1})}\exp\big(-8^{L \beta}\,c\,[\log\log\frac{1}{r}]^{L \beta}\big).$
\vskip0.3cm
By Lemma \ref{lem2.4} it follows that if
\begin{equation}\label{eq3.9}
\eta\,(1-\eta) \xi\,\omega(r) \geqslant r
\end{equation}
then
$$
u(y)\leqslant\max\limits_{B_{\frac{r}{2}}(y)} u \leqslant M(r)- (1-\eta)^{2}\xi\,\omega(r)
=\frac{11}{16}\,u(y),
$$
reaching a contradiction, which proves the claim.

We note that inequality \eqref{eq3.9} is a consequence of \eqref{eq3.6}, provided that $\rho$ is small enough and
\begin{equation}\label{eq3.10}
L\, \beta <1 \quad \text{and} \quad l_{4} \geqslant \gamma(l_{1}).
\end{equation}

We set
$$\alpha(r) := \nu_{1}\,[\log\log\frac{1}{r}]^{-\bar{\beta}_{2}}\,\exp\bigg(-\frac{q}{p\kappa} 8^{L\beta}\,c\,[\log\log\frac{1}{r}]^{L \beta}\bigg)$$
and  apply Theorem \ref{th1.5} with $\alpha = \alpha(r)$, $\bar{r}=\varepsilon_{0} \alpha^{2}(r) r [\log\log\frac{1}{r}]^{-L}$  and
$N=\frac{u(y)}{4}$,
note that by our choices $\bar{r}=\psi(r)$, so we obtain
\begin{multline*}
\min\limits_{B_{\frac{\rho}{2}}(x_{0})} u \geqslant \gamma^{-1}\,u(y)\,\alpha^{\tau_{1}}(r)\bigg(\frac{r}{\rho}\bigg)^{\tau_{2}}
\,\exp\bigg(-c\int\limits_{2\bar{r}}^{\rho}[\log\log\frac{1}{s}]^{L\,\beta}\frac{ds}{s}\bigg) \geqslant \\ \geqslant
\gamma^{-1}\nu^{\tau_{1}}_{1}\,u(y)\bigg(\frac{r}{\rho}\bigg)^{\tau_{2}}[\log\log\frac{1}{r}]^{-\bar{\beta}_{2} \tau_{1}}
\,\exp\bigg(-c\int\limits_{\psi(2r)}^{\rho}[\log\log\frac{1}{s}]^{L \beta}\,\frac{ds}{s}-c\frac{q}{p\kappa}8 ^{L\beta}[\log\log\frac{1}{r}]^{L\beta} \bigg).
\end{multline*}
If $\rho$ is small enough, then $\log\log\dfrac{1}{r} \leqslant \dfrac{\rho}{r} \log\log\dfrac{1}{\rho}$, hence
from the previous we obtain
\begin{multline*}
\min\limits_{B_{\frac{\rho}{2}}(x_{0})} u \geqslant \gamma^{-1}\,\nu^{\tau_{1}}_{1}\,u_{0}\,\bigg(\frac{\rho}{r}\bigg)^{l_{1}-1-\tau_{2}-\bar{\beta}_{2}\tau_{1}}\,[\log\log\frac{1}{\rho}]^{-\bar{\beta}_{2}\tau_{1}}\exp\big(-l_{2}[\log\log\frac{1}{\rho}]^{L l_{3}}\big)\times\\ \times\exp\bigg(l_{2}[\log\log\frac{1}{r}]^{L l_{3}}-c\frac{q}{p\kappa}8 ^{L\beta}[\log\log\frac{1}{r}]^{L\beta} \bigg)\, \exp\big(-c\int\limits_{\psi(\rho)}^{\rho}[\log\log\frac{1}{s} ]^{L \beta}\frac{ds}{s}\big).
\end{multline*}
Fix $l_{1}, l_{2}$ and $l_{3}$ by the conditions
\begin{equation*}
l_{1}=1+\tau_{2} +\bar{\beta}_{2}\tau_{1},\quad l_{2}=c\frac{q}{p\kappa}8^{L\beta},\quad l_{3}=\beta,
\end{equation*}
then the last inequality can be rewritten as
\begin{equation*}
\min\limits_{B_{\frac{\rho}{2}}(x_{0})} u \geqslant \gamma^{-1}\,\nu^{\tau_{1}}_{1}\,u_{0}\,[\log\log\frac{1}{\rho}]^{-\bar{\beta}_{2}\tau_{1}}\exp\bigg(-l_{2}[\log\log\frac{1}{\rho}]^{L l_{3}}-c\int\limits_{\psi(\rho)}^{\rho}[\log\log\frac{1}{s} ]^{L \beta}\frac{ds}{s}\bigg).
\end{equation*}
Using the fact that $\log\log\frac{1}{\psi(\rho)}\leqslant \gamma \log\log\frac{1}{\rho}$, $\log\frac{\rho}{\psi(\rho)}\leqslant \gamma [\log\log\frac{1}{\rho}]^{L \beta}$ and $[\log\log\frac{1}{\rho}]^{\bar{\beta}_{2}\tau_{1}}\leqslant\\ \leqslant \exp\big(\gamma [\log\log\frac{1}{\rho}]^{L\beta}\big)$  if $\rho$ is small enough, from the previous we arrive at
\begin{equation*}
\min\limits_{B_{\frac{\rho}{2}}(x_{0})} u \geqslant \gamma^{-1}\,\nu^{\tau_{1}}_{1}\,u_{0}\,\exp\big(-\gamma [\log\log\frac{1}{\rho}]^{L\beta}\big).
\end{equation*}

Choose $l_{4}$ sufficiently large, according to \eqref{eq3.10}, choose $L$ sufficiently small, according to \eqref{eq3.7}, \eqref{eq3.10}, choose $\rho$ small enough, according to the second inequality in \eqref{eq3.7} and note that  \eqref{eq3.8} holds by our choice of $\alpha(r)$,  we arrive at
\begin{equation*}
\min\limits_{B_{\frac{\rho}{2}}(x_{0})} u \geqslant \gamma^{-1}\,\nu^{\tau_{1}}_{1}\,\frac{u_{0}}{\log\frac{1}{\rho}},
\end{equation*}
provided that inequality \eqref{eq3.6} holds. This completes the proof of Theorem \ref{th1.3}.

\section{ Pointwise Estimates of Auxiliary Solutions, Proof of \\Theorems \ref{th1.4} and \ref{th1.6}}\label{Sec4}

We will assume that the following
integral identity holds:
\begin{equation}\label{eq4.1}
\int\limits_{D} \varPhi(x, |\nabla w|)\,\frac{\nabla w}{|\nabla w|^{2}} \nabla\eta \,dx=0
\quad \text{for any } \ \eta \in W^{1,\varPhi}_{0}(D).
\end{equation}
The existence of the solutions $w$ follows from the general theory of monotone operators. Testing \eqref{eq4.1} by $\eta=(w-m)_{+}$ and  by $\eta=w_{-}$  we obtain that $0\leqslant w\leqslant m \leqslant \lambda(\rho) M$.

To formulate our next result, we need the notion of the capacity. For this set
$$
C_{\varPhi}(E, B_{8\rho}(x_{0});m):=\,\frac{1}{m}\,\inf\limits_{v\in \mathfrak{M}(E)}
\int\limits_{B_{8\rho}(x_{0})}\varPhi(x,m|\nabla v|)\,dx ,
$$
where the infimum is taken over the set $\mathfrak{M}(E)$   of all functions
$v \in W^{1,\varPhi}_{0}(B_{8\rho}(x_{0}))$ with $v \geqslant1$ on $E$. If $m=1$,
this definition leads to the standard definition of $C_{\varPhi}(E, B_{8\rho}(x_{0}))$
capacity (see, e.g. \cite{HarHasZAn19}).

Further we will assume that
\begin{equation}\label{eq4.2}
\Lambda^{-1}_{+,x_{0},8\rho}\bigg(\frac{C_{\varPhi}(E, B_{8\rho}(x_{0}); m)}{\rho^{n-1}\,[\Lambda_{\lambda}\big(x_{0},\rho\big)]^{\bar{c}_{1}}}\bigg) \geqslant \bar{c},
\end{equation}
where $\Lambda^{-1}_{+,x_{0},8\rho}(\cdot)$ is the inverse function to $\Lambda_{+,\varphi}(x_{0}, 8\rho, \cdot)$, $\varphi(x, v):=\dfrac{\varPhi(x, v)}{v}$, $v>0$ and $\bar{c}, \bar{c}_{1} > 0 $ to be chosen later depending only on the data.

\subsection{Upper bound for the function $w$}\label{subsec4.1}
We note that in the standard case (i.e. if $p=q$) the upper bound for the function $w$
was proved in \cite{IVSkr1983}
(see also \cite[Chap.\,8, Sec.\,3]{IVSkrMetodsAn1994}, \cite{IVSkrSelWorks}).
\begin{lemma}\label{lem4.1}
There exists $\bar{\beta} >0$ depending only on the data such that
\begin{equation*}
w(x)\leqslant \gamma\,\,\,\rho\,\Lambda^{-1}_{+,x_{0},8\rho}\bigg(\big[\Lambda_{\lambda}(x_{0}, \rho)\big]^{\bar{\beta}}\,\frac{C_{\varPhi}(E, B_{8\rho}(x_{0}); m)}{\rho^{n-1}}\bigg),\quad x\in K_{\frac{3}{2}\rho, 8\rho}=B_{8\rho}(x_{0})\setminus B_{\frac{3}{2}\rho}(x_{0}).
\end{equation*}
\end{lemma}
\begin{proof}
Fix $\sigma\in(0,\frac{1}{8})$, $s\in ((1+\sigma)\rho, (8-\sigma)\rho)$ and  let $M_{0}:=\sup\limits_{K_{s,8\rho}} w$, $M_{\sigma}:= \sup\limits_{K_{s(1+\sigma), 8\rho}} w$. Fix $\bar{x}\in K_{s(1+\sigma), 8\rho}$ and for $j= 0,1,2,....$ set $\rho_{j}:=s\frac{\sigma}{2}(1+  2^{-j})$, $B_{j}:=B_{\rho_{j}}(\bar{x})$, $k_{j}= k(1-2^{-j})$, where $k >0$    and suppose that $(w(\bar{x})-k)_{+} \geqslant \bar{c} \rho\big[\Lambda_{\lambda}(x_{0}, \rho)\big]^{\bar{c}_{1}}$, then $\sup\limits_{B_{j}}(w-k_{j})_{+} \geqslant \rho$. We will use inequality \eqref{eq2.2}, by condition $(\varPhi^{\lambda}_{\Lambda, x_{0}})$ with $K=M$, using the fact that $\sup\limits_{B_{j}}(w-k_{j})_{+} \leqslant \\ \leqslant m\leqslant \lambda(\rho) M$ we have
\begin{equation*}
\Lambda_{\varPhi}\big(\bar{x}, \rho_{j}, \sup\limits_{B_{j}}(w-k_{j})_{+}/\rho_{j}\big)\leqslant \gamma(M) \Lambda_{\lambda}(\bar{x}, \rho_{j})\leqslant \gamma \Lambda_{\lambda}(x_{0}, \rho).
\end{equation*}
Hence, inequality \eqref{eq2.2} with $\theta= p\frac{t}{t+1}$ can be rewritten as
\begin{equation*}
\int\limits_{B_{j+1}\cap\{u>k_{j}\}}|\nabla w|^{\theta}\,dx \leqslant \frac{\gamma}{\sigma^{\gamma}} 2^{j\gamma}\, \bigg(\frac{M_{0}}{\rho}\bigg)^{\theta}\,\rho^{n}  [\Lambda_{\lambda}(x_{0}, \rho)]^{\frac{\theta}{p}}\bigg(\frac{|B_{j}\cap\{w>k_{j}\}|}{|B_{j}|}\bigg)^{1-\frac{\theta}{tp} -\frac{\theta}{sp}}.
\end{equation*}
Since $\bar{x}$ is an arbitrary point in $\bar{x}\in K_{s(1+\sigma), 8\rho}$ this inequality by standard arguments ( see e.g. \cite{LadUr}) yields
\begin{equation*}
M_{\sigma}^{\theta+\frac{1}{\kappa}} \leqslant \gamma\,\sigma^{-\gamma}\,M_{0}^{\frac{1}{\kappa}}\,
[\Lambda_{\lambda}(x_{0}, \rho)]^{\frac{1}{p\kappa}}\,\rho^{-n}\int\limits_{K_{s,8\rho}}\,w^{\theta}\,dx +\gamma\,\sigma^{-\gamma}\rho^{\theta+\frac{1}{\kappa}}\big[\Lambda_{\lambda}(x_{0}, \rho)\big]^{(\theta+\frac{1}{\kappa})\bar{c}_{1}},
\end{equation*}
where $\kappa=\frac{1}{n}-\frac{1}{tp}-\frac{1}{sp} >0$. Using the Young inequality we obtain for any $\varepsilon \in(0, 1)$
\begin{equation}\label{eq4.3}
M_{\sigma} \leqslant \varepsilon M_{0} +\gamma\,\sigma^{-\gamma}\,\varepsilon^{-\gamma}[\Lambda_{\lambda}(x_{0}, \rho)]^{\frac{1}{p\theta \kappa}}\,\bigg(\rho^{-n}\int\limits_{K_{s,8\rho}}\,w^{\theta}\,dx\bigg)^{\frac{1}{\theta}} +\gamma\,\rho\big[\Lambda_{\lambda}(x_{0},\rho)\big]^{\bar{c}_{1}},\quad \theta=p\frac{t}{t+1}.
\end{equation}
Let us estimate the second term on the right-hand side of \eqref{eq4.3}. For this we assume that \\ $M_{\sigma}\geqslant \bar{c}\sigma^{-\gamma}\rho\big[\Lambda_{\lambda}(x_{0}, \rho)\big]^{\bar{c}_{1}} \geqslant \rho$, because otherwise, by \eqref{eq4.2} the upper estimate  is evident. Set \\$w_{M_{0}}:=\min\{w, M_{0}\}$, by the Poincare and H\"{o}lder inequalities we have with arbitrary $\varepsilon_{1} \in (0, 1)$
\begin{multline}\label{eq4.4}
\bigg(\rho^{-n}\int\limits_{K_{s, 8\rho}} w^{\theta}\,dx\bigg)^{\frac{1}{\theta}} = \bigg(\rho^{-n}\int\limits_{K_{s, 8\rho}} w^{\theta}_{M_{0}}\,dx\bigg)^{\frac{1}{\theta}}\leqslant \gamma \rho \bigg(\rho^{-n}\int\limits_{D}
| \nabla w_{M_{0}}|^{\theta}\,dx\bigg)^{\frac{1}{\theta}} \leqslant \\ \leqslant \gamma \rho \bigg(\rho^{-n}\int\limits_{D}|\nabla w_{M_{0}}|^{p}\varphi_{p}\big(x,\varepsilon_{1}\frac{M_{0}}{\rho}\big)\,dx\bigg)^{\frac{1}{p}}\bigg(\rho^{-n}\int\limits_{B_{8\rho}(x_{0})} \big[\varphi_{p}\big(x, \varepsilon_{1}\frac{M_{0}}{\rho}\big)\big]^{-t}\,dx\bigg)^{\frac{1}{pt}}=\\=\gamma \varepsilon_{1} M_{0} \bigg(\rho^{-n}\int\limits_{D}|\nabla w_{M_{0}}|^{p}\varphi_{p}\big(x,\varepsilon_{1}\frac{M_{0}}{\rho}\big)\,dx\bigg)^{\frac{1}{p}}\bigg(\rho^{-n}\int\limits_{B_{8\rho}(x_{0})} \big[\varPhi\big(x, \varepsilon_{1} \frac{M_{0}}{\rho}\big)\big]^{-t}\,dx\bigg)^{\frac{1}{pt}} \leqslant \\ \leqslant
\gamma \varepsilon_{1} M_{0} \bigg(\rho^{-n}\int\limits_{D}\varPhi\big(x,|\nabla w_{M_{0}}|\big)dx +\rho^{-n}\int\limits_{B_{8\rho}(x_{0})}\varPhi\big(x,\varepsilon_{1} \frac{M_{0}}{\rho}\big)\,dx\bigg)^{\frac{1}{p}}\times \\ \times \bigg(\rho^{-n}\int\limits_{B_{8\rho}(x_{0})} \big[\varPhi\big(x, \varepsilon_{1} \frac{M_{0}}{\rho}\big)\big]^{-t}\,dx\bigg)^{\frac{1}{pt}},
\end{multline}
above we also used inequality \eqref{eq2.1} with $\varepsilon=1$. Fix $\varepsilon_{1}$ by the condition
$$\varepsilon_{1}= \gamma^{-1} \varepsilon^{1+\gamma} \sigma^{\gamma} \big[\Lambda_{\lambda}(x_{0}, \rho)\big]^{-\frac{1}{p}(1+\frac{1}{\theta \kappa})}.$$
 The second term on the right-hand side of \eqref{eq4.4} we estimate using
condition $(\varPhi^{\lambda}_{\Lambda, x_{0}})$. By our choice $\varepsilon_{1} M_{0} \geqslant \varepsilon_{1}\bar{c} \sigma^{-\gamma} \big[\Lambda(x_{0},\rho)\big]^{\bar{c}_{1}}\rho \geqslant \rho$, provided that $\bar{c} \geqslant \varepsilon^{-1-\gamma}\,\gamma$ and $\bar{c}_{1} \geqslant \frac{1}{p}(1+\frac{1}{\theta \kappa})$, moreover $\varepsilon_{1}M_{0}\leqslant m \leqslant  \lambda(\rho) M$, therefore by condition $(\varPhi^{\lambda}_{\Lambda, x_{0}})$
\begin{multline}\label{eq4.5}
\bigg(\rho^{-n}\int\limits_{B_{8\rho}(x_{0})}\big[\varPhi\big(x,\varepsilon_{1} \frac{M_{0}}{\rho}\big)\big]^{s}\,dx\bigg)^{\frac{1}{ps}}\bigg(\rho^{-n}\int\limits_{B_{8\rho}(x_{0})} \big[\varPhi\big(x, \varepsilon_{1} \frac{M_{0}}{\rho}\big)\big]^{-t}\,dx\bigg)^{\frac{1}{pt}} \leqslant \\
 \leqslant \gamma \big[\Lambda_{+,\varPhi}(x_{0}, 8\rho,\varepsilon_{1} \frac{M_{0}}{\rho})\big]^{\frac{1}{p}}\,\,\big[\Lambda_{-,\varPhi}(x_{0}, 8\rho,\varepsilon_{1}\frac{ M_{0}}{\rho})\big]^{\frac{1}{p}}
\leqslant \gamma \big[\Lambda_{\lambda}(x_{0},\rho)\big]^{\frac{1}{p}}.
\end{multline}
Combining estimates \eqref{eq4.3}--\eqref{eq4.5} and using condition $(\varPhi)$ we arrive at
\begin{multline}\label{eq4.6}
M_{\sigma}\leqslant 2\varepsilon M_{0} +\gamma M_{0}\frac{\big[\Lambda_{\lambda}(x_{0},\rho)\big]^{\frac{1}{p}}}{\big[\Lambda_{+,\varPhi}(x_{0}, 8\rho,\varepsilon_{1} \dfrac{M_{0}}{\rho})\big]^{\frac{1}{p}}}\bigg(\rho^{-n}\int\limits_{D} \varPhi(x, |\nabla w_{M_{0}}| \,dx \bigg)^{\frac{1}{p}}
+\gamma \rho \big[\Lambda_{\lambda}(x_{0}, \rho)\big]^{\bar{c}_{1}} \leqslant \\
\leqslant 2\varepsilon M_{0} +\gamma \varepsilon^{-\gamma}\sigma^{-\gamma} M_{0}\frac{\big[\Lambda_{\lambda}(x_{0},\rho)\big]^
{\bar{\beta}_{4}}}{\big[\Lambda_{+,\varPhi}(x_{0}, 8\rho, \dfrac{M_{0}}{\rho})\big]^{\frac{1}{p}}}\bigg(\rho^{-n}\int\limits_{D} \varPhi(x, |\nabla w_{M_{0}}| \,dx \bigg)^{\frac{1}{p}}+\gamma \rho \big[\Lambda_{\lambda}(x_{0}, \rho)\big]^{\bar{c}_{1}},
\end{multline}
where $\bar{\beta}_{4}=\frac{1}{p}(1+(q-1)\frac{1}{\theta\kappa})$.

We need to estimate the integral on the right-hand side of \eqref{eq4.6}. Let $\psi \in W^{1,\varPhi}_{0}(B_{8\rho}(x_{0}))$, $\psi=1 $ on $E$, be such that $\frac{1}{m}\int\limits_{B_{8\rho}(x_{0})}\varPhi(x, m|\nabla \psi|)\,dx \leqslant \gamma C_{\varPhi}(E, B_{8\rho}(x_{0});m) +\gamma \rho^{n}$, test identinty \eqref{eq4.1} by $ \eta= w- m\psi$, using inequality \eqref{eq2.1} with $\theta=1$, $a=|\nabla w|$,
$b=|\nabla \psi|$ and sufficiently small $\varepsilon$  we obtaint
\begin{equation*}
\int\limits_{D}\varPhi(x,|\nabla w|)\,dx \leqslant \gamma \int\limits_{B_{8\rho}(x_{0})} \varPhi(x,m |\nabla \psi|)\,dx \leqslant
\gamma m C_{\varPhi}(E, B_{8\rho}(x_{0});m) +\gamma m\rho^{n}.
\end{equation*}
Now testing identity \eqref{eq4.1} by $\eta= w_{M_{0}} -\dfrac{M_{0}}{m}w$, using \eqref{eq2.1} and the previous inequality we obtain
\begin{equation}\label{eq4.7}
\int\limits_{D}\varPhi(x,|\nabla w_{M_{0}}|)\,dx \leqslant \gamma \frac{M_{0}}{m}\int\limits_{D}\varPhi(x,|\nabla w|)\,dx \leqslant \gamma M_{0} C_{\varPhi}(E, B_{8\rho}(x_{0});m) +\gamma M_{0}\rho^{n}.
\end{equation}
Combining estimates \eqref{eq4.6} and \eqref{eq4.7}, using the fact that $\Lambda_{+,\varPhi}(x_{0}, 8\rho, \frac{M_{0}}{\rho})=\frac{M_{0}}{\rho} \Lambda_{+,\varphi}(x_{0}, 8\rho, \frac{M_{0}}{\rho})$  and using our assumption that $M_{0} \geqslant \bar{c} \rho\big[\Lambda_{\lambda}(x_{0}, \rho)\big]^{\bar{c}_{1}}$ we obtain
\begin{multline*}
M_{\sigma}\leqslant 2\varepsilon M_{0} +\gamma \varepsilon^{-\gamma}\sigma^{-\gamma} M_{0}\frac{\big[\Lambda_{\lambda}(x_{0},\rho)\big]^
{\bar{\beta}_{4}}}{\big[\Lambda_{+,\varPhi}(x_{0}, 8\rho, \frac{M_{0}}{\rho})\big]^{\frac{1}{p}}}\bigg(\rho^{-n} M_{0} C_{\varPhi}(E, B_{8\rho}(x_{0});m) + M_{0} \bigg)^{\frac{1}{p}} + \\+\gamma \rho \big[\Lambda_{\lambda}(x_{0}, \rho)\big]^{\bar{c}_{1}}  \leqslant  (2\varepsilon+ \frac{\gamma}{\bar{c}}) M_{0} +\gamma \varepsilon^{-\gamma}\sigma^{-\gamma} M_{0}\frac{\big[\Lambda_{\lambda}(x_{0},\rho)\big]^
{\bar{\beta}_{4}}}{\big[\Lambda_{+,\varphi}(x_{0}, 8\rho, \frac{M_{0}}{\rho})\big]^{\frac{1}{p}}}\bigg(\rho^{1-n} C_{\varPhi}(E, B_{8\rho}(x_{0});m) + \rho \bigg)^{\frac{1}{p}} \\ \leqslant (2\varepsilon +\frac{\gamma}{\bar{c}}) M_{0} +\gamma \varepsilon^{-\gamma}\sigma^{-\gamma} M_{0}\frac{\big[\Lambda_{\lambda}(x_{0},\rho)\big]^{\bar{\beta}_{4}}}{\big[\Lambda_{+,\varphi}(x_{0}, 8\rho, \frac{M_{0}}{\rho})\big]^{\frac{1}{p}}}\bigg(\rho^{1-n} C_{\varPhi}(E, B_{8\rho}(x_{0});m)\bigg)^{\frac{1}{p}} .
\end{multline*}
By the fact that  $\Lambda_{+,\varphi}(x_{0}, \rho, \varepsilon_{0} \frac{M_{0}}{\rho})\leqslant \varepsilon^{p-1}_{0}
\Lambda_{+,\varphi}(x_{0}, \rho, \frac{M_{0}}{\rho})$, $\varepsilon_{0} \in (0, 1)$ and using \eqref{eq2.1} with $a= \frac{M_{\sigma}}{\rho}$, $b=1$, $\varepsilon$ replaced by $\varepsilon_{0} \frac{M_{0}}{\rho}$ and  $\varphi_{p}(x,\cdot)$ replaced by  $\big[\Lambda_{+,\varphi}(x_{0}, \rho, \cdot )\big]^{\frac{1}{p}}$ from this we obtain
\begin{multline*}
\big[\Lambda_{+,\varphi}(x_{0}, 8\rho, \frac{M_{\sigma}}{\rho})\big]^{\frac{1}{p}}\leqslant \frac{M_{\sigma}}{\varepsilon_{0} M_{0}}
\big[\Lambda_{+,\varphi}(x_{0}, 8\rho, \frac{M_{\sigma}}{\rho})\big]^{\frac{1}{p}}
+\big[\Lambda_{+,\varphi}(x_{0}, 8\rho, \varepsilon_{0} \frac{M_{0}}{\rho})\big]^{\frac{1}{p}} \leqslant\\ \leqslant  \frac{M_{\sigma}}{\varepsilon_{0} M_{0}}\big[\Lambda_{+,\varphi}(x_{0}, 8\rho, \frac{M_{0}}{\rho})\big]^{\frac{1}{p}}+\varepsilon^{\frac{p-1}{p}}_{0}\big[\Lambda_{+,\varphi}(x_{0}, 8\rho, \frac{M_{0}}{\rho})\big]^{\frac{1}{p}} \leqslant\\ \leqslant  (\varepsilon^{\frac{p-1}{p}} +2\frac{\varepsilon}{\varepsilon_{0}}+\frac{\gamma}{\bar{c}})\big[\Lambda_{+,\varphi}(x_{0}, 8\rho, \frac{M_{0}}{\rho})\big]^{\frac{1}{p}} + \gamma \varepsilon^{-1}_{0}\varepsilon^{-\gamma}\sigma^{-\gamma} \big[\Lambda_{\lambda}(x_{0},\rho)\big]^{\bar{\beta}_{4}} \bigg(\rho^{1-n} C_{\varPhi}(E, B_{8\rho}(x_{0});m)\bigg)^{\frac{1}{p}}.
\end{multline*}
Fix $\varepsilon$, $\bar{c}$ by the conditions $\varepsilon = \frac{1}{4} \varepsilon^{2+\frac{p}{p-1}}_{0}$, $\bar{c} \geqslant \frac{4\gamma}{\varepsilon_{0}}$,  $\varepsilon_{0} \in (0, 1)$, then
the last inequality yields
\begin{equation*}
\big[\Lambda_{+,\varphi}(x_{0}, 8\rho, \frac{M_{\sigma}}{\rho})\big]^{\frac{1}{p}}\leqslant \varepsilon_{0}\big[\Lambda_{+,\varphi}(x_{0}, 8\rho, \frac{M_{0}}{\rho})\big]^{\frac{1}{p}} + \gamma \varepsilon^{-\gamma}_{0} \sigma^{-\gamma}  \big[\Lambda_{\lambda}(x_{0},\rho)\big]^{\bar{\beta}_{4}} \bigg(\rho^{1-n} C_{\varPhi}(E, B_{8\rho}(x_{0});m)\bigg)^{\frac{1}{p}}.
\end{equation*}
Iterating this inequality we arrive at
\begin{equation*}
\big[\Lambda_{+,\varphi}(x_{0}, 8\rho, \frac{M_{0}}{\rho})\big]^{\frac{1}{p}}\leqslant \gamma \big[\Lambda_{\lambda}(x_{0},\rho)\big]^{\bar{\beta}_{4}} \bigg(\rho^{1-n} C_{\varPhi}(E, B_{8\rho}(x_{0});m)\bigg)^{\frac{1}{p}},
\end{equation*}
which proves the required upper bound with $\bar{\beta}=p \bar{\beta}_{4}= 1+(q-1)\frac{t+1}{pt \kappa}$. This completes the proof of the lemma.
\end{proof}

\subsection{Lower bound for the function $w$}\label{subsec4.2}

The main step in the proof of the lower bound is the following lemma.
\begin{lemma}\label{lem4.2}
There exist  numbers $\varepsilon$, $\vartheta\in(0,1)$, $\bar{\beta}_{5} >0$ depending only on the data
such that
\begin{multline}\label{eq4.8}
\left| \left\{K_{\frac{3}{2}\rho,\, 4\rho}:
w(x)\leqslant \varepsilon\,\rho\,
\Lambda^{-1}_{+, x_{0},8\rho}\left( \frac{C_{\varPhi}(E, B_{8\rho}(x_{0});m)}{\rho^{\,n-1}} \right)\right\} \right|
\leqslant \\ \leqslant \left(1-\vartheta\,\big[\Lambda_{\lambda}(x_{0}, \rho)\big]^{-\bar{\beta}_{5}} \right) |K_{\frac{3}{2}\rho,\, 4\rho}|.
\end{multline}
\end{lemma}
\begin{proof}
Let $\zeta_{1}(x) \in C_{0}^{\infty}(B_{3\rho}(x_{0}))$, $0\leqslant \zeta_{1}(x)\leqslant 1$, $\zeta_{1}(x)=1$ in $B_{2\rho}(x_{0})$ and $|\nabla \zeta_{1}|\leqslant \frac{\gamma}{\rho}.$ Test \eqref{eq4.1} by $\eta= w- m\zeta^{q}$ and use the Young inequality \eqref{eq2.1} with $\theta=1$ we obtain
with any $\varepsilon_{1} \in(\rho, M \lambda(\rho))$
\begin{multline*}
\int\limits_{D}\,\varPhi(x, |\nabla w|)\,dx \leqslant \gamma \frac{m}{\rho} \int\limits_{K_{2\rho, 3\rho}}\varphi(x,|\nabla w|)\,\zeta^{q-1}_{1}\,dx \leqslant \gamma\frac{m}{\varepsilon_{1}}\int\limits_{K_{2\rho, 3\rho}}\varPhi(x,|\nabla w|)\,dx +\\+
\gamma\frac{m}{\rho}\int\limits_{K_{2\rho, 3\rho}}\varphi(x,\frac{\varepsilon_{1}}{\rho})\,dx \leqslant \gamma\frac{m}{\varepsilon_{1}}\int\limits_{K_{2\rho, 3\rho}}\varPhi(x,|\nabla w|)\,dx + \gamma m\rho^{n-1}\bigg(\rho^{-n}\int\limits_{B_{8\rho}(x_{0})}\big[\varphi(x,\frac{\varepsilon_{1}}{\rho})\big]^{s}\,dx\bigg)^{\frac{1}{s}}\leqslant \\ \leqslant \gamma\frac{m}{\varepsilon_{1}}\int\limits_{K_{2\rho, 3\rho}}\varPhi(x,|\nabla w|)\,dx + \gamma m\rho^{n-1} \Lambda_{+,\varphi}(x_{0}, 8\rho, \frac{\varepsilon_{1}}{\rho}).
\end{multline*}
Let $\zeta_{2}(x)\in C^{\infty}_{0}(K_{\frac{3}{2}\rho, 4\rho})$, $0\leqslant \zeta_{2}(x)\leqslant 1$, $\zeta_{2}(x) =1 $ in $K_{\rho, 2\rho}$
and $| \nabla \zeta_{2}|\leqslant \frac{\gamma}{\rho}$. Testing \eqref{eq4.1} by $\eta= w\,\zeta^{q}$ and using the Young inequality \eqref{eq2.1} we estimate the first term on the right-hand side of the previous inequality as follows
\begin{equation*}
\int\limits_{K_{2\rho, 3\rho}}\varPhi(x,|\nabla w|)\,dx \leqslant \int\limits_{K_{\frac{3}{2}\rho, 4\rho}}\varPhi(x,|\nabla w|)\,\zeta^{q}_{2}\,dx \leqslant \gamma \int\limits_{K_{\frac{3}{2}\rho, 4\rho}}\varPhi\big(x,\frac{w}{\rho}\big)\,dx.
\end{equation*}
Combining the last two inequalities and using the definition of the capacity we obtain
\begin{equation}\label{eq4.9}
C_{\varPhi}(E, B_{8\rho}(x_{0});m) \leqslant \frac{1}{m} \int\limits_{D}\varPhi(x,|\nabla w|)\,dx \leqslant
\frac{\gamma}{\varepsilon_{1}}\int\limits_{K_{\frac{3}{2}\rho, 4\rho}}\varPhi\big(x,\frac{w}{\rho}\big)\,dx +
\gamma \rho^{n-1} \Lambda_{+,\varphi}(x_{0}, 8\rho, \frac{\varepsilon_{1}}{\rho}).
\end{equation}
Choose $\varepsilon_{1}$ by the condition $\gamma \rho^{n-1} \Lambda_{+,\varphi}(x_{0}, 8\rho, \frac{\varepsilon_{1}}{\rho}) = \frac{1}{4} C_{\varPhi}(E, B_{8\rho}(x_{0});m)$, i.e.
\begin{equation*}
\varepsilon_{1}= \rho \Lambda^{-1}_{+,x_{0}, 8\rho}\bigg(\frac{C_{\varPhi}(E, B_{8\rho}(x_{0});m)}{4\gamma\,\rho^{\,n-1}} \bigg)
\geqslant (4\gamma)^{-\frac{1}{p-1}}\,\,\rho\,\Lambda^{-1}_{+,x_{0}, 8\rho}\bigg(\frac{C_{\varPhi}(E, B_{8\rho}(x_{0});m)}{\rho^{\,n-1}} \bigg).
\end{equation*}
Hence inequality \eqref{eq4.9} implies
\begin{equation}\label{eq4.10}
C_{\varPhi}(E, B_{8\rho}(x_{0});m) \leqslant \frac{\gamma}{\varepsilon_{1}}\int\limits_{K_{\frac{3}{2}\rho, 4\rho}}\varPhi\big(x,\frac{w}{\rho}\big)\,dx.
\end{equation}
Let us estimate the integral on the right-hand side of inequality \eqref{eq4.10}, for this we decompose $K_{\frac{3}{2}\rho, 4\rho}$ as
$K_{\frac{3}{2}\rho, 4\rho} = K'\cup K''$, $K':=K_{\frac{3}{2}\rho, 4\rho}\cap\big\{w\leqslant \varepsilon\,\rho\,\delta(\rho, m) \big\}$ and  $K'':=K_{\frac{3}{2}\rho, 4\rho}\setminus K'$, $\delta(\rho, m)=\Lambda^{-1}_{+,x_{0},8\rho}\left( \frac{C_{\varPhi}(E, B_{8\rho}(x_{0});m)}{\rho^{\,n-1}} \right)$. By $(\varPhi)$ and our choice of $\varepsilon_{1}$ we have
\begin{multline}\label{eq4.11}
\frac{\gamma }{\varepsilon_{1}}\int\limits_{K'}\varPhi\big(x,\frac{w}{\rho}\big)\,dx \leqslant \frac{\gamma \varepsilon^{p}}{\varepsilon_{1}}\delta(\rho, m)\,\rho^{n}\bigg(\rho^{-n} \int\limits_{B_{8\rho}(x_{0})} \big[\varphi(x, \delta(\rho, m))\big]^{s}\,dx\bigg)^{\frac{1}{s}}
\leqslant \\ \leqslant \gamma \varepsilon^{p}\rho^{n-1} \Lambda_{+, \varphi}(x_{0}, 8\rho, \delta(\rho, m)) \leqslant \gamma \varepsilon^{p}
C_{\varPhi}(E, B_{8\rho}(x_{0}) ;m).
\end{multline}
Similarly, by $(\varPhi)$, Lemma \ref{lem4.1} and our choice of $\varepsilon_{1}$
\begin{multline}\label{eq4.12}
\frac{\gamma }{\varepsilon_{1}}\int\limits_{K''}\varPhi\big(x,\frac{w}{\rho}\big)\,dx\leqslant \frac{\gamma}{\varepsilon_{1}}\delta(\rho, m)\,\rho^{n}[\Lambda_{\lambda}(x_{0},\rho)]^{\frac{q\bar{\beta}}{p}}\bigg(\rho^{-n} \int\limits_{B_{8\rho}(x_{0})} \big[\varphi(x, \delta(\rho, m))\big]^{s}\,dx\bigg)^{\frac{1}{s}}\bigg(\frac{|K''|}{|K_{\frac{3}{2}\rho, 4\rho}|}\bigg)^{1-\frac{1}{s}}\\
\leqslant \gamma C_{\varPhi}(E, B_{8\rho}(x_{0}) ;m)[\Lambda_{\lambda}(x_{0},\rho)]^{\frac{q\bar{\beta}}{p}}\bigg(\frac{|K''|}{|K_{\frac{3}{2}\rho, 4\rho}|}\bigg)^{1-\frac{1}{s}}.
\end{multline}
Collecting estimates \eqref{eq4.10}--\eqref{eq4.12} we obtain
\begin{equation*}
1\leqslant \gamma \varepsilon^{p} + \gamma [\Lambda_{\lambda}(x_{0},\rho)]^{\frac{q\bar{\beta}}{p}}\bigg(\frac{|K''|}{|K_{\frac{3}{2}\rho, 4\rho}|}\bigg)^{1-\frac{1}{s}},
\end{equation*}
choosing $\varepsilon$ from the condition $\gamma \varepsilon^{p} = \frac{1}{2}$, from the previous we arrive at the required
\eqref{eq4.8} with $\bar{\beta}_{5}= \frac{s}{s-1} \frac{q\bar{\beta}}{p}$, which completes the proof of the lemma.
\end{proof}
The following lemma is the main result of this Paragraph.
\begin{lemma}\label{lem4.3}
There exist numbers $\bar{\varepsilon}\in (0,1)$, $\bar{\beta}_{6}$, $\bar{\beta}_{7} > 0$ depending only on the data such that
\begin{equation}\label{eq4.13}
\big|\big\{K_{\frac{3}{2}\rho, 4\rho} : w \leqslant \bar{\varepsilon} m \big[\Lambda_{\lambda}(x_{0}, \rho)\big]^{-\bar{\beta}_{6}} \bigg(\frac{|E|}{\rho^{n}}\bigg)^{\bar{\beta}_{7}} \big\}\big|\leqslant \big(1-\vartheta\,\big[\Lambda_{\lambda}(x_{0}, \rho)\big]^{-\bar{\beta}_{5}} \big) |K_{\frac{3}{2}\rho,\, 4\rho}|,
\end{equation}
provided that
\begin{equation}\label{eq4.14}
m\,\bigg(\frac{|E|}{\rho^{n}}\bigg)^{\bar{\beta}_{7}} \geqslant \frac{\rho}{\bar{\varepsilon}} \big[\Lambda_{\lambda}(x_{0}, \rho)\big]^{\bar{\beta}_{6}},
\end{equation}
where $\bar{\beta}_{5}$, $\vartheta > 0$ were defined in Lemma \ref{lem4.2}.
\end{lemma}
\begin{proof}
To prove inequality \eqref{eq4.13} we need to estimate the term on the left-hand side of \eqref{eq4.8}. Let $\psi\in W^{1,\varPhi}_{0}(B_{8\rho}(x_{0}))$,  $\psi(x)=1$ for $x\in E$ and fix  $\theta \in (1+ \frac{1}{t}, p)$. By the Poincare and H\"{o}lder inequalities and using \eqref{eq2.1} with  $\varepsilon \in (0,1)$ we have
\begin{multline}\label{eq4.15}
m |E| \leqslant m\,\int\limits_{B_{8\rho}(x_{0})} | \psi|\,dx \leqslant \gamma\,\rho \int\limits_{B_{8\rho}(x_{0})} |\nabla( m\psi)|\,dx
\leqslant\\ \leqslant \gamma\,\rho \bigg(\int\limits_{B_{8\rho}(x_{0})} |\nabla( m\psi)|^{\theta}\,\varphi_{\theta}(x, \frac{m}{\rho})\,dx \bigg)^{\frac{1}{\theta}} \bigg(\int\limits_{B_{8\rho}(x_{0})}\big[\varphi_{\theta}(x,\frac{m}{\rho})\big]^{-\frac{1}{\theta-1}}\,dx\bigg)^{1-\frac{1}{\theta}}\leqslant\\ \leqslant \gamma m \rho^{n} \bigg(\varepsilon^{-\theta}\rho^{-n}\int\limits_{B_{8\rho}(x_{0})}\varPhi(x, |\nabla(m\psi)|)\,dx + \varepsilon^{p-\theta}\bigg(\rho^{-n} \int\limits_{B_{8\rho}(x_{0})} \big[\varPhi(x,\frac{m}{\rho})\big]^{s}\,dx\bigg)^{\frac{1}{s}}\bigg)^{\frac{1}{\theta}}\times \\ \times \bigg(\rho^{-n}\int\limits_{B_{8\rho}(x_{0})}\big[\varPhi(x,\frac{m}{\rho})\big]^{-t}\,dx\bigg)^{\frac{1}{\theta t}}.
\end{multline}
Let us estimate the terms on the right-hand side of \eqref{eq4.15}. Since $\rho \leqslant m\leqslant \lambda(\rho) M$, by $(\varPhi^{\lambda}_{\Lambda, x_{0}})$ we have
\begin{equation}\label{eq4.16}
\bigg(\rho^{-n} \int\limits_{B_{8\rho}(x_{0})} \big[\varPhi(x,\frac{m}{\rho})\big]^{s}\,dx\bigg)^{\frac{1}{\theta s}}
\bigg(\rho^{-n}\int\limits_{B_{8\rho}(x_{0})}\big[\varPhi(x,\frac{m}{\rho})\big]^{-t}\,dx\bigg)^{\frac{1}{\theta t}} \leqslant \gamma
[\Lambda_{\lambda}(x_{0},\rho)]^{\frac{1}{\theta}},
\end{equation}
therefore, choosing $\varepsilon$ from the condition
\begin{equation*}
\gamma \rho^{n} \varepsilon^{\frac{p-\theta}{\theta}}[\Lambda_{\lambda}(x_{0},\rho)]^{\frac{1}{\theta}} =\frac{1}{2}|E|,\quad \text{i.e.}\quad
\varepsilon=(2\gamma)^{-\frac{\theta}{p-\theta}}[\Lambda_{\lambda}(x_{0},\rho)]^{-\frac{1}{p-\theta}}\bigg(\frac{|E|}{\rho^{n}}\bigg)^{\frac{\theta}{p-\theta}} < 1,
\end{equation*}
we obtain from \eqref{eq4.15}, \eqref{eq4.16}
\begin{equation*}
\gamma^{-1}\,\varepsilon^{\theta} \bigg(\frac{|E|}{\rho^{n}}\bigg)^{\theta} \leqslant \frac{\Lambda_{\lambda}(x_{0},\rho)}{\bigg(\rho^{-n} \int\limits_{B_{8\rho}(x_{0})} \big[\varPhi(x,\frac{m}{\rho})\big]^{s}\,dx\bigg)^{\frac{1}{ s}}}\,\,\rho^{-n}\int\limits_{B_{8\rho}(x_{0})}\varPhi(x, |\nabla(m\psi)|)\,dx.
\end{equation*}
From this, using the definition of capacity and using the fact that $\varPhi(x,v)= v \varphi(x,v)$, $v>0$, we obtain
\begin{equation*}
\Lambda_{+,\varphi}(x_{0}, 8\rho,\frac{m}{\rho})\bigg(\frac{|E|}{\rho^{n}}\bigg)^{\theta}\leqslant \gamma \varepsilon^{-\theta}
\Lambda_{\lambda}(x_{0},\rho)\rho^{1-n}
C_{\varPhi}(E, B_{8\rho}(x_{0}); m),
\end{equation*}
which yields by our choice of $\varepsilon$
\begin{multline}\label{eq4.17}
\rho \Lambda^{-1}_{+,x_{0},8\rho}\bigg(\frac{C_{\varPhi}(E, B_{8\rho}(x_{0}); m)}{\rho^{n-1}}\bigg) \geqslant
\gamma^{-1}\,\,m\,\,\bigg(\varepsilon^{\theta} \big[\Lambda_{\lambda}(x_{0},\rho)\big]^{-1}\bigg(\frac{|E|}{\rho^{n}}\bigg)^{\theta}\bigg)^{\frac{1}{p-1}}=\\= \gamma^{-1}\,\,m\,\, \big[\Lambda_{\lambda}(x_{0},\rho)\big]^{-\bar{\beta}_{6}}
\bigg(\frac{|E|}{\rho^{n}}\bigg)^{\bar{\beta}_{7}},\quad \bar{\beta}_{6}=\frac{p}{(p-\theta)(p-1)},
\bar{\beta}_{7}=\frac{p\theta}{(p-\theta)(p-1)}.
\end{multline}
Therefore, inequality \eqref{eq4.13} is a consequence of \eqref{eq4.8}, provided that \eqref{eq4.2} is valid. By \eqref{eq4.17}, inequality \eqref{eq4.2}, in turn, is a consequence of  \eqref{eq4.14}, provided that $\bar{c}$ is large enough and $\bar{c}_{1} \geqslant \bar{\beta}_{6}$. This completes the proof of the lemma.

\subsection{Proof of Theorems \ref{th1.4} and \ref{th1.6}}\label{Subsec4.3}
Let $u \geqslant 0 $ be a super-solution to Eq. \eqref{eq1.14} in $\Omega$ and construct the sets $E(\rho, N):=B_{\rho}(x_{0})\cap\\ \cap \{u > N\}$ and $E_{\lambda}(\rho, N):=B_{\rho}(x_{0})\cap \{u > \lambda(\rho)\,N\}$, $ 0<N < M$, $E(\rho, N)\subset E_{\lambda}(\rho, N)$. Let $w$ be an auxiliary solution to the  problem \eqref{eq1.24} in $D=B_{8\rho}\setminus E_{\lambda}(\rho, N)$. Since $u \geqslant w$ on
$\partial D$, by the monotonicity condition \eqref{eq1.16} $u\geqslant w$ in $D$ and Lemma \ref{lem4.3} with $m=\lambda(\rho)N$ implies
\begin{equation}\label{eq4.18}
\big|\big\{B_{2\rho}(x_{0}) : u \leqslant \bar{\varepsilon} \lambda(\rho) N\big[\Lambda_{\lambda}(x_{0}, \rho)\big]^{-\bar{\beta}_{6}} \bigg(\frac{|E_{\lambda}(\rho)|}{\rho^{n}}\bigg)^{\bar{\beta}_{7}} \big\}\big|\leqslant \big(1-\vartheta\,
\big[\Lambda_{\lambda}(x_{0}, \rho)\big]^{-\bar{\beta}_{5}} \big) |B_{2\rho}(x_{0})|,
\end{equation}
provided that
\begin{equation}\label{eq4.19}
\lambda(\rho)\,N\,\bigg(\frac{|E_{\lambda}(\rho,m)|}{\rho^{n}}\bigg)^{\bar{\beta}_{7}} \geqslant \frac{\rho}{\bar{\varepsilon}} \big[\Lambda_{\lambda}(x_{0}, \rho)\big]^{\bar{\beta}_{6}},
\end{equation}
with some positive $\bar{\beta}_{5}$, $\bar{\beta}_{6}$, $\bar{\beta}_{7}$ and $\vartheta \in (0,1)$ depending only on the data.

We use Lemmas \ref{lem2.3}, \ref{lem2.4} with $\alpha_{0}= \vartheta \big[\Lambda_{\lambda}(x_{0}, \rho)\big]^{-\bar{\beta}_{5}}$ and $\nu$ defined in Lemma \ref{lem2.4}. These lemmas ensure the existence of $C$, $\bar{\beta}_{8}= \bar{\beta}_{1}+
p\bar{\beta}_{1}\bar{\beta}_{5} +p \bar{\beta}_{1}\bar{\beta}_{2} > 0$, where $\bar{\beta}_{1}$, $\bar{\beta}_{2} >0$ were defined in Lemmas \ref{lem2.3} and \ref{lem2.4} depending only on the data, such that
\begin{multline*}
\lambda(\rho) N \bigg(\frac{|E(\rho,m)|}{\rho^{n}}\bigg)^{\bar{\beta}_{7}} \leqslant \lambda(\rho) N \bigg(\frac{|E_{\lambda}(\rho,m)|}{\rho^{n}}\bigg)^{\bar{\beta}_{7}} \leqslant \\ \leqslant \big[\Lambda_{\lambda}(x_{0}, \rho)\big]^{-\bar{\beta}_{6}}\exp\big(C\big[\Lambda_{\lambda}(x_{0},\rho)\big]^{\bar{\beta}_{8}}\big) \{ \inf\limits_{B_{\frac{\rho}{2}}(x_{0})} u  + \rho \}\leqslant\\ \leqslant \exp\big(C\big[\Lambda_{\lambda}(x_{0},\rho)\big]^{1+\bar{\beta}_{8}}\big) \{ \inf\limits_{B_{\frac{\rho}{2}}(x_{0})} u  + \rho \}.
\end{multline*}
This completes the proof of Theorem \ref{th1.6}.

The proof of the weak Harnack-type inequality \eqref{eq1.17} and the upper bound \eqref{eq1.18} is almost the same as for Theorem \ref{th1.2}, inequalities \eqref{eq1.11} and \eqref{eq1.12}, see Section \ref{subsect3.3} for details, we leave them  to the reader.

\end{proof}

\vskip3.5mm
{\bf Acknowledgements.} This work is supported by grants of the National Academy of Sciences of Ukraine (project numbers is 0120U100178) and by the Volkswagen Foundation project "From Modeling and Analysis to Approximation".

\bigskip

CONTACT INFORMATION

\medskip
\textbf{Maria O.~Savchenko}\\Institute of Applied Mathematics and Mechanics,
National Academy of Sciences of Ukraine, \\ \indent Batiouk Str. 19, 84116 Sloviansk, Ukraine\\
Vasyl' Stus Donetsk National University,
\\ \indent 600-richcha Str. 21, 21021 Vinnytsia, Ukraine\\shan\underline{ }maria@ukr.net

\medskip
\textbf{Igor I.~Skrypnik}\\Institute of Applied Mathematics and Mechanics,
National Academy of Sciences of Ukraine, \\ \indent Batiouk Str. 19, 84116 Sloviansk, Ukraine\\
Vasyl' Stus Donetsk National University,
\\ \indent 600-richcha Str. 21, 21021 Vinnytsia, Ukraine\\ihor.skrypnik@gmail.com

\medskip
\textbf{Yevgeniia A. Yevgenieva}
\\ Max Planck Institute for Dynamics of Complex Technical Systems, \\ \indent Sandtorstrasse 1, 39106 Magdeburg, Germany
\\Institute of Applied Mathematics and Mechanics,
National Academy of Sciences of Ukraine, \\ \indent Batiouk Str. 19, 84116 Sloviansk, Ukraine\\yevgeniia.yevgenieva@gmail.com

\end{document}